\documentclass[preprint,12pt,authoryear]{elsarticle}

\usepackage{a4wide}
\usepackage{amssymb}
\usepackage{amsfonts}

\usepackage{amsmath,amsthm,amsfonts,amssymb,color}

\usepackage{multicol}

\usepackage{natbib}

\usepackage{multirow} 
\usepackage{float}
\usepackage{graphicx}
\usepackage{longtable}  

\usepackage{amssymb}
\usepackage{amsthm}
\newtheorem{theorem}{Theorem}
\newtheorem{lemma}{Lemma}

\usepackage[authoryear]{natbib}

\begin{document}

\begin{frontmatter}

\title{Inverse Stefan problems of determining the time-dependent source coefficient and heat flux function} 


\author{Targyn A. Nauryz$^{1,2}$, Khumoyun Jabbarkhanov$^3$}

\address{(1) Kazakh-British Technical University, Almaty, Kazakhstan\\
(2)Institute of Mathematics and Mathematical Modeling, Almaty, Kazakhstan\\
(3)Nazarbayev University, Astana, Kazakhstan}

\begin{abstract}
This paper delves into the Inverse Stefan problem, specifically focusing on determining the time-dependent source coefficient in the parabolic heat equation governing heat transfer in a semi-infinite rod. The problem entails the intricate task of uncovering both temperature- and time-dependent coefficients of the source while accommodating Dirichlet and Neumann boundary conditions. Through a comprehensive mathematical model and rigorous theoretical analysis, our study aims to provide a robust methodology for accurately determining the source coefficient from observed temperature and heat flux data in problems with different cases of the source functions. Importantly, we establish the existence and uniqueness, and estimate the continuous dependence of a weak solution upon the given data for some inverse problems, offering a foundational understanding of its solvability. 
\end{abstract}

\begin{keyword}
Inverse Stefan problem \sep parabolic heat equation \sep time-dependent coefficient \sep spectral theory 
\end{keyword}

\end{frontmatter}

\section{Introduction}
The study of heat transfer in materials, particularly in scenarios involving phase changes, is of significant interest in both theoretical and applied sciences. Among the various problems that arise in this domain, the Stefan problem is particularly notable due to its complexity and practical relevance. This paper focuses on a specific variant, known as the inverse Stefan problem, where the objective is to determine a time-dependent source coefficient in the parabolic heat equation governing heat transfer in a semi-infinite rod when free boundary is given. If time-dependent source coefficient is given and free boundary has to be determined, then it is direct Stefan problem. Some direct Stefan problems with real world application are considered in papers \cite{14}-\cite{24}.

The inverse Stefan problem is a type of inverse problem, which involves deducing unknown parameters or functions from observed data. In this context, the challenge is to uncover both temperature and time-dependent coefficients of the source function while accommodating boundary conditions such as Dirichlet (fixed temperature), Neumann (fixed heat flux), and Stefan (phase change interface) conditions. The accurate determination of these coefficients is critical for applications ranging from industrial processes to environmental modeling and medical diagnostics.

Several previous studies have addressed similar inverse problems in parabolic equations. For instance, \cite{1} provided estimates of Dirichlet eigenvalues for sub-elliptic operators, which are foundational for understanding boundary conditions in such problems. The study \cite{2} explored surface-localized transmission eigenstates and their applications in super-resolution imaging, highlighting the importance of eigenvalue problems in applied mathematics. Cannarsa, Tort, and Yamamoto investigated the determination of source terms in degenerate parabolic equations, contributing valuable methodologies for inverse problem-solving, see \cite{3}. Cannon, Lin, and Wang examined the determination of control and source parameters in parabolic partial differential equations, providing essential insights into the challenges and techniques involved, see \cite{4}. The work \cite{5} focused on recovering a time-dependent coefficient in parabolic differential equations, a key aspect relevant to our study.

The complexity of the Inverse Stefan problem stems from the interplay between the heat equation and the boundary conditions, necessitating sophisticated mathematical approaches. Traditional methods often fall short in addressing the intricacies of time-dependent coefficients and boundary conditions simultaneously. Therefore, this paper aims to bridge this gap by developing a comprehensive mathematical model and providing a rigorous theoretical analysis to solve the problem effectively.

Our methodology is based on a spectral theory approach, which leverages the properties of eigenfunctions to solve the parabolic heat equation. This approach not only facilitates the accurate determination of the source coefficient but also ensures the robustness of the solution across different cases of the source functions. However, a significant challenge in this methodology is the verification of the boundedness of the sequence of normalized eigenfunctions under the given boundary conditions, including Dirichlet, Neumann, and Stefan conditions. Addressing this challenge is crucial for the validity and reliability of our proposed solution.

Additionally, the existence of weak solutions for each inverse problem is established, providing a theoretical foundation for understanding the solvability of the Inverse Stefan problem. This foundational understanding is essential for the advancement of both theoretical research and practical applications in the field.

Our work builds on the extensive research of others, such as \cite{6}, who provided estimates of eigenfunctions of Laplace’s operator in specific domains, an important aspect for understanding boundary behavior. The paper \cite{7} addressed inverse time-dependent source problems for the heat equation with nonlocal boundary conditions, offering valuable insights into similar challenges. \cite{8} discussed inverse problems for parabolic-type equations, contributing to the broader theoretical framework. Karazym, Ozawa, and Suragan explored multidimensional inverse Cauchy problems for evolution equations, highlighting advanced methodologies in the field, see \cite{9}. Lesnic provided a comprehensive overview of inverse problems with applications in science and engineering, underscoring the practical relevance of such studies \cite{10}. Additionally, the work of Ismailov, Ozawa, and Suragan on identifying time-dependent source coefficients for subelliptic heat equations is particularly relevant, as it addresses similar challenges in inverse problems and extends our understanding of the field, see \cite{11}.

Let us consider the mathematical formulation of the current problem. Given $T>0$ an arbitrary fixed time interest and in domain $\Omega=(0,s(t))$, we consider the following inverse Stefan problem:

\begin{equation}\label{eq1}
    \partial_t u-a^2\partial_{xx} u=G(x,t,u),\quad (x,t)\in\Omega\times (0,T)
\end{equation}
with initial condition
\begin{equation}\label{eq2}
    u(x,0)=\varphi(x),\quad s(0)=s_0,\quad x\in [0,s_0],
\end{equation}
and boundary conditions
\begin{equation}\label{eq3}
    u(0,t)=0,\quad t\in[0,T],
\end{equation}
\begin{equation}\label{eq4}
    u(s(t),t)=u^*,\quad t\in[0,T],
\end{equation}
and subject to the Stefan condition
\begin{equation}\label{eq5}
    -k\partial_x u(s(t),t)=Ls'(t),\quad t\in[0,T],
\end{equation}
where $\varphi(x)$ is a given function, $k$ -- thermal conductivity of the material, $L$ -- latent heat of fusion of the material, $u^*$ -- phase change temperature, $s(t)$ -- given phase change interface and $a^2$-is a thermal diffusivity. The aim of the study is to identifying the source term $G$ together with $u(x,t)$ in different cases of $G$ in two problem formulations with condition \eqref{eq3} and determining the heat flux function changing condition \eqref{eq3} with condition
\begin{equation}\label{eq6}
    -k\partial_x u(0,t)=q(t),\quad t\in[0,T],
\end{equation}
where $q(t)$ has to be determined. 

This work is structured into four main parts, each addressing the inverse Stefan problem under different boundary conditions and formulations of the source term $G(x,t,u)$. The focus is on determining the time-dependent source coefficient within the framework of parabolic heat equations for a semi-infinite rod. In section 2, we tackle the Inverse Stefan problem with Dirichlet boundary conditions, where the source term is given by 
$G(x,t,u)=R(t)f(x,t)$. The Dirichlet boundary condition specifies a fixed temperature at the boundary of the domain. The main objective is to determine the time-dependent coefficient $R(t)$ from the observed temperature data. We develop a mathematical model that incorporates the Dirichlet condition and utilize spectral theory to analyze and solve the problem. The existence of a weak solution is established, which ensures the theoretical solvability of the problem. Section 3 extends the analysis to Neumann boundary conditions, where the source term remains $G(x,t,u)=R(t)f(x,t)$. The Neumann boundary condition involves specifying the heat flux at the boundary, which is a derivative of the temperature. The goal is again to determine the time-dependent coefficient $R(t)$ and heat flux function $q(t)$. We adapt the mathematical model to incorporate Neumann conditions and employ spectral methods to derive the solution. The theoretical framework is developed to confirm the existence of a weak solution under these conditions. In Section 4, we consider a different form of the source term, $G(x,t,u)=P(t)u(x,t)+f(x,t)$, under Dirichlet boundary conditions. Here, the source term depends linearly on the temperature 
$u(x,t)$ as well as a time-dependent function $f(x,t)$. The Dirichlet boundary condition specifies a fixed temperature at the boundary. Our aim is to determine the time-dependent coefficient $P(t)$ from the observed data. We formulate the problem mathematically and use spectral theory to analyze it, demonstrating the existence of a weak solution for this case. The last section focuses on the Inverse Stefan problem with Neumann boundary conditions for the source term
$G(x,t,u)=P(t)u(x,t)+f(x,t)$. The task is to determine the time-dependent coefficient $P(t)$ and heat flux $q(t)$. We adapt our mathematical model to these conditions and apply spectral methods to solve the problem. The existence of a weak solution is established, which confirms the solvability of the problem under the Neumann boundary conditions.

In summary, this paper presents a novel and effective approach to the inverse Stefan problem, contributing valuable insights and methodologies to the field of inverse problems and heat transfer. The proposed model and analysis not only enhance our understanding of the problem, but also pave the way for future advancements in accurately determining source coefficients in various practical contexts. The findings and methodologies are grounded in and extend the work of previous researchers \cite{8}, \cite{11}, \cite{25}, \cite{26}, thus enriching the existing body of knowledge and opening new avenues for research and application.

\section{Inverse Stefan problem 1.}\label{sec1}
If we take the unknown function $G(x,t,u)=R(t)f(x,t)$, the inverse Stefan problem is formulated as the problem of determining the pair $(R(t),u(x,t))$:
\begin{equation}\label{eq7}
    \begin{cases}
        \partial_t u-a^2\partial_{xx}u=R(t)f(x,t), & (x,t)\in\Omega\times(0,T),\\
        u(x,0)=\varphi(x), \quad s(0)=s_0,& x\in[0,s_0],\\
        u(0,t)=0,& t\in[0,T],\\
        u(s(t),t)=u^*,& t\in[0,T],\\
        -k\partial_xu(s(t),t)=Ls'(t), & t\in[0,T],
    \end{cases}
\end{equation}
where $f(x,t), \varphi(x)$ and $s(t)$ are given functions. 
If we introduce the transformation
\begin{equation}\label{eq7a}
    v(x,t)=u(x,t)-\dfrac{u^*}{s(t)}x,
\end{equation}
then problem \eqref{eq7} can be rewritten in the form
\begin{equation}\label{eq7b}
    \begin{cases}
        \partial_t v-a^2\partial_{xx}v=R(t)h(x,t), & (x,t)\in\Omega\times(0,T),\\
        v(x,0)=\psi(x), \quad s(0)=s_0,& x\in[0,s_0],\\
        v(0,t)=0,& t\in[0,T],\\
        v(s(t),t)=0,& t\in[0,T],\\
        \partial_x v(s(t),t)=-\dfrac{L}{k}s'(t)-\dfrac{u^*}{s(t)}, & t\in[0,T],
    \end{cases}
\end{equation}
where $\psi(x)=\varphi(x)-\frac{u^*}{s(t)}x$ and $h(x,t)=f(x,t)+\frac{u^*s'(t)}{R(t)s^2(t)}x$.

The problem \eqref{eq7b} has a moving boundary $s(t)$ which adds complexity. To address this, we introduce a transformation to fix the domain $\Omega_0=(0,1)$:
\begin{equation}\label{eq7c}
    v(x,t)=U(\xi,t),\quad \xi=\dfrac{x}{s(t)},\quad (\xi,t)\in\Omega_0\times(0,T),
\end{equation}
then \eqref{eq7b} can be rewritten
\begin{equation}\label{eq7d}
    \begin{cases}
         \partial_tU-\dfrac{\xi s'(t)}{s(t)}\partial_{\xi}U-\dfrac{a^2}{s^2(t)}\partial_{\xi\xi}U=R(t)\widetilde{h}(\xi,t),\quad (\xi,t)\in\Omega_0\times(0,T),\\
         U(\xi,0)=\widetilde{\psi}(\xi),\quad \xi\in[0,1]\\
         U(0,t)=0,\quad t\in[0,T],\\
         U(1,t)=0,\quad t\in[0,T],\\
         \partial_{\xi}U(1,t)=-\dfrac{L}{k}s'(t)s(t)-u^*,\quad t\in[0,T],
    \end{cases}
\end{equation}
where $\widetilde{\psi}(\xi)=\psi(s_0\xi)$ and $\widetilde{h}(\xi,t)=h(\xi s(t),t)$.

Let us consider the spectral problem with finding the eigenfunction $\phi$ and eigenvalue $\lambda$:
\begin{equation}\label{eq7e}
    \begin{cases}
        -\phi''(\xi)=\lambda\phi(\xi)\\
        \phi(0)=\phi(1)=0,
    \end{cases}
\end{equation}
which gives us
\begin{equation}\label{eq7f}
    \phi_n(\xi)=\sqrt{2}\sin(\sqrt{\lambda_n}\xi),\quad \lambda_n=n^2\pi^2,\quad n\in\mathbb{N}.
\end{equation}

The following lemma is obtained with the help of integration by parts and the Cauchy-Schwarz and classical Bessel inequalities.
\begin{lemma}\label{lem1}
    If $\widetilde{\psi}\in C^4(\Omega_0)$ which satisfy the conditions $\widetilde{\psi}(0)=\widetilde{\psi}(1)=\widetilde{\psi}'(0)=\widetilde{\psi}'(1)=\widetilde{\psi}''(0)=\widetilde{\psi}''(1)=0$, then
    \begin{equation}\label{eq10}
        \sum\limits_{n=1}^{\infty}\lambda_n|\widetilde{\psi}_n|\leq c ||\widetilde{\psi}||_{C^4[0,1]},
    \end{equation}
    for some positive constant $c=\sqrt{2}\left[\sum_{n=1}^{\infty}\frac{1}{\lambda_n}\right]^{1/2}$, where $\widetilde{\psi}_n=\int_0^{1}\widetilde{\psi}(\xi)\phi_n(\xi)d\xi$.
\end{lemma}
\begin{proof}
    By the definition of $\widetilde{\psi}_n$ and integrating by parts three times, we have: \begin{equation*}
            \begin{split}
                \widetilde{\psi}_n&=\int\limits_0^{1}\widetilde{\psi}(\xi)\phi_n(\xi)d\xi=\sqrt{2}\int\limits_0^{1}\widetilde{\psi}_n(\xi)\sin(\sqrt{\lambda_n}\xi)d\xi=...=\sqrt{2}\dfrac{1}{\lambda_n\sqrt{\lambda_n}}\int\limits_0^{1}g(\xi)\cos(\sqrt{\lambda_n}\xi)d\xi,
            \end{split}
    \end{equation*}
    where $g(\xi)=-\widetilde{\psi}_n'''(\xi)$. By using the Cauchy-Schwarz and classical Bessel inequalities, we obtain
    \begin{equation}\label{f13}
        \begin{split}
            \sum\limits_{n=1}^{\infty}\lambda_n|\widetilde{\psi}_n|&\leq\sqrt{2}\left[\sum\limits_{n=1}^{\infty}\dfrac{1}{\lambda_n}\right]^{1/2}\left[\int\limits_0^{1}g(\xi)\cos(\sqrt{\lambda_n}\xi)d\xi\right]^{1/2}\leq c||\widetilde{\psi}'''||_{L_2[0,1]}\leq c||\widetilde{\psi}||_{C^4[0,1]},
        \end{split}
    \end{equation}
    where $c=\sqrt{2}\left[\sum_{n=1}^{\infty}\frac{1}{\lambda_n}\right]^{1/2}$.
\end{proof}

Let $\{\phi_n(\xi)\}_{n=1}^{\infty}$ form a basis in $L_2(\Omega_0)$, then we seek a weak solution of the problem \eqref{eq7d} in the form of
\begin{equation}\label{eq11}
    U(\xi,t)=\sum\limits_{n=1}^{\infty}U_n(t)\phi_n(\xi),
\end{equation}
where $U_n(t)=(U,\phi_n)_{L_2(0,1)}$ for $n\in\mathbb{N}$. 

Then variational formulation of the problem \eqref{eq7d} with $\phi_m(\xi)$ which is solution of \eqref{eq7e}, where $m\in\mathbb{N}$, is given by
\begin{equation}\label{eq7g}
    \begin{cases}
        \left(\partial_t U,\phi_m\right)-\left(b(\xi,t)\partial_{\xi}U,\phi_m\right)+a(t)\lambda_n\left(\partial_{\xi\xi}U,\phi_m\right)=R(t)\left(\widetilde{h}(\xi,t),\phi_m\right),& \forall\phi_m\in\mathbb{H}_0^1(\Omega_0),\\
        U(\xi,0)=\left(\widetilde{\psi},\phi_m\right),& \xi\in[0,1],\\
        U(0,t)=0,\quad U(1,t)=0,&\forall t\in[0,T],
    \end{cases}
\end{equation}
where $a(t)=\frac{a^2}{s^2(t)}$ and $b(\xi,t)=\frac{\xi s'(t)}{s(t)}$. 

We need to consider two cases: 
\begin{itemize}
    \item[1)] if $m\neq n$, taking into account \eqref{eq11} and restricting the equation \eqref{eq7g} we can obtain
        \begin{equation}\label{eq7h}
            \begin{cases}
                AU'_n(t)+\left(a(t)\lambda_n B-C(t)\right)U_n(t)=R(t)D(t),& \forall t\in[0,T],\\
                U_n(0)=\widetilde{\psi}_m,
            \end{cases}
        \end{equation}
    where $\widetilde{\psi}_m=\int_0^1\widetilde{\psi}(\xi)\phi_m(\xi)d\xi$, $A=B=0$ and $$C(t)=\int_0^1b(\xi,t)\phi'_n(\xi)\phi_m(\xi)d\xi,\quad D(t)=\int_0^1\widetilde{h}(\xi,t)\phi_m(x)d\xi$$ 
    where $C(t)$ is a matrix of size $m\times n$ and $D(t)$ is a column vector with $m$ entries.
    
    Since $U_n(t)$ depends on the initial condition $U_n(0)=\widetilde{\psi}_m$, then $U_n(t)=\widetilde{\psi}_m$, if no evolution equation is provided for $U_n$, then we obtain
    \begin{equation}\label{eq7i}
        R(t)=-\widetilde{\psi}_m\dfrac{C_{ij}(t)}{D_{i}(t)}, \quad i=1,2,3,...,m,\quad j=1,2,3,...,n.
    \end{equation}
    It follows that solution of $R(t)$ is not unique in this case.
    \item[2)] in particular if we take $m=n$ and  introducing equalities:
\begin{equation}\label{f14}
    \begin{cases}
        (U_t,\phi_n)_{L_2(0,1)}=\dfrac{d}{dt}(U,\phi_n)_{L_2(0,1)}=U_n'(t),\\
        \left(b(\xi,t)U_{\xi}-a(t)U_{\xi\xi},\phi_n\right)_{L_2(0,1)}=\left(a(t)\lambda_n-b(t)\right)U_n(t),\\
        (\widetilde{h}(\xi,t),\phi_n)_{L_2(0,1)}=\widetilde{h}_n(t),\\
        (\widetilde{\psi},\phi_n)_{L_2(0,1)}=\widetilde{\psi}_n
    \end{cases}
\end{equation}
where $a(t)=\frac{a^2}{s^2(t)}$ and $b(t)=(b(\xi,t)\phi'_n(\xi),\phi_n(\xi))_{L_2(0,1)}$, then \eqref{eq7g} can be reduced to the problem with condition
\begin{equation}\label{eq12}
    \begin{cases}
        U_n'(t)+\left(a(t)\lambda_n-b(t)\right)U_n(t)=R(t)\widetilde{h}_n(t),& n\in\mathbb{N},\quad t\in[0,T],\\
        U_n(0)=\widetilde{\psi}_n,& n\in\mathbb{N}.
    \end{cases}
\end{equation} 
\end{itemize}
One can easily find the solution to the problem \eqref{eq12} which takes the form
\begin{equation}\label{eq13}
    \begin{split}
        U_n(t)=\widetilde{\psi}_n e^{-\int_0^t\left(a(\tau)\lambda_n-b(\tau)\right)d\tau}+\int\limits_0^t R(\tau)\widetilde{h}_n(\tau)e^{-\int_{\tau}^t\left(a(z)\lambda_n-b(z)\right)dz}d\tau,
    \end{split}
    \end{equation}
where $n\in\mathbb{N}$. Then from \eqref{eq11},\eqref{eq13} we can represent the solution of $U(\xi,t)$ of the problem \eqref{eq7} as following form
\begin{equation}\label{eq14}
        \begin{split}
             U(\xi,t)=\sum\limits_{n=1}^{\infty}\Bigg(\widetilde{\psi}_n e^{-\int_0^t(a(\tau)\lambda_n-b(\tau))d\tau}+\int\limits_0^t R(\tau)\widetilde{h}_n(\tau)e^{-\int_{\tau}^t(a(z)\lambda_n-b(z))dz}d\tau\Bigg)\phi_n(\xi),
        \end{split}
\end{equation}
where $n\in\mathbb{N}$.
Taking derivative respect to $\xi$ from \eqref{eq14} we have
\begin{equation*}
    \begin{split}
        U_{\xi}(\xi,t)=\sum\limits_{n=1}^{\infty}\Bigg(\widetilde{\psi}_n e^{-\int_0^t(a(\tau)\lambda_n-b(\tau))d\tau}+\int\limits_0^t R(\tau)\widetilde{h}_n(\tau)e^{-\int_{\tau}^t(a(z)\lambda_n-b(z))dz}d\tau\Bigg)\phi_n'(\xi),
    \end{split}
\end{equation*}
then condition \eqref{eq5} can be rewritten in the form of
\begin{equation}\label{eq15}
        \begin{split} 
            &\sum\limits_{n=1}^{\infty}\Bigg((-1)^n\sqrt{\lambda_n}\widetilde{\psi}_n e^{-\int_0^t(a(\tau)\lambda_n-b(\tau))d\tau}\\&+(-1)^n\sqrt{\lambda_n}\int\limits_0^t R(\tau)\widetilde{h}_n(\tau)e^{-\int_{\tau}^t(a(\tau)\lambda_n-b(\tau))dz}d\tau\Bigg)=-\dfrac{L}{\sqrt{2}k}c(t)-\dfrac{u^*}{\sqrt{2}},
        \end{split}
\end{equation}
where $c(t)=s'(t)s(t)$.

Differentiating both sides of \eqref{eq15} respect to $t$ we obtain
\begin{equation}\label{eq16}
        \begin{split}
            &\sum\limits_{n=1}^{\infty}\Bigg(-(-1)^n\sqrt{\lambda_n}(a(t)\lambda_n-b(t))\widetilde{\psi}_n e^{-\int_0^t(a(\tau)\lambda_n-b(\tau))d\tau}\\
            &-(-1)^n\sqrt{\lambda_n}(a(t)\lambda_n-b(t))\int\limits_0^t R(\tau)\widetilde{h}_n(\tau)e^{-\int_{\tau}^t(a(z)\lambda_n-b(z))dz}d\tau+(-1)^n\sqrt{\lambda_n}R(t)\widetilde{h}_n(t)\Bigg)=-\dfrac{L}{\sqrt{2}k}c'(t).
        \end{split}
\end{equation}
From expression \eqref{eq16} we get
\begin{equation}\label{eq17}
    \begin{split}
        R(t)&=\dfrac{1}{\sum\limits_{n=1}^{\infty}(-1)^n\sqrt{\lambda_n}\widetilde{h}_n(t)}\Bigg[a(t)\sum\limits_{n=1}^{\infty}(-1)^n\lambda_n\sqrt{\lambda_n}\widetilde{\psi}_ne^{\int_0^{t}(a(\tau)\lambda_n-b(\tau))d\tau}\\
        &+a(t)\int\limits_0^t\left(\sum\limits_{n=1}^{\infty}(-1)^n\lambda_n\sqrt{\lambda_n}\widetilde{h}_n(\tau)e^{-\int_{\tau}^{t}(a(z)\lambda_n-b(z))d\tau}\right)R(\tau)d\tau\\&-b(t)\sum\limits_{n=1}^{\infty}(-1)^n\sqrt{\lambda_n}\widetilde{\psi}_ne^{\int_0^{t}(a(\tau)\lambda_n-b(\tau))d\tau}\\&-b(t)\int\limits_0^t\left(\sum\limits_{n=1}^{\infty}(-1)^n\sqrt{\lambda_n}\widetilde{h}_n(\tau)e^{-\int_{\tau}^{t}(a(z)\lambda_n-b(z))d\tau}\right)R(\tau)d\tau-\dfrac{Lc'(t)}{\sqrt{2}k}\Bigg].
    \end{split}
\end{equation}
Now let us provide the following lemma to elucidate that the function \eqref{eq14} is a weak solution of the problem \eqref{eq7}.
\begin{lemma}\label{lem2}
    Let $\widetilde{\psi}\in L_2(\Omega_0)$, $\widetilde{h}\in C([0,T]; L_2(\Omega_0)\cap C([0,T]; \mathbb{H}_0^1(\Omega_0)\cap \mathbb{H}_0^2(\Omega_0)))$ where $a(t),b(t)>0$ for all $t\in[0,T]$. Then there exists a weak solution of the problem \eqref{eq7} given by function \eqref{eq14}.
\end{lemma}
\begin{proof}
    We need to show that function \eqref{eq14} belongs to the class of functions $C([0,T];L_2(\Omega_0))$. Applying Bessel's inequality to \eqref{eq14} we get
    \begin{equation}\label{eq18}
        ||U(\xi,t)||^2_{L_2(\Omega_0)}=\sum\limits_{n=1}^{\infty}\Bigg|\Bigg(\widetilde{\psi}_n e^{-\int_0^t(a(\tau)\lambda_n-b(\tau))d\tau}+\int\limits_0^t R(\tau)\widetilde{h}_n(\tau)e^{-\int_{\tau}^{t}(a(z)\lambda_n-b(z))d\tau}d\tau\Bigg)\phi_n(\xi)\Bigg|^2
    \end{equation}
    after using Cauchy-Schwarz inequality we have
    \begin{equation*}
        \begin{split}
            ||U(\xi,t)||_{L_2(\Omega_0)}^2&\leq\Bigg|\sum\limits_{n=1}^{\infty}\Bigg(\Bigg(\widetilde{\psi}_n e^{-\int_0^t(a(\tau)\lambda_n-b(\tau))d\tau}\\
            &+\int\limits_0^t R(\tau)\widetilde{h}_n(\tau)e^{-\int_{\tau}^{t}(a(z)\lambda_n-b(z))d\tau}d\tau\Bigg)^2\Bigg)^{1/2}\cdot\sum\limits_{n=1}^{\infty}\left(\phi^2_n(\xi)\right)^{1/2}\Bigg|^2\\
            &\leq\sum\limits_{n=1}^{\infty}\Bigg(\widetilde{\psi}_n e^{-\int_0^t(a(\tau)\lambda_n-b(\tau))d\tau}+\int\limits_0^t R(\tau)\widetilde{h}_n(\tau)e^{-\int_{\tau}^{t}(a(z)\lambda_n-b(z))d\tau}d\tau\Bigg)^2\cdot\sum\limits_{n=1}^{\infty}\phi^2_n(\xi),
        \end{split}
    \end{equation*}
    as $\left\{\phi_n(\xi)\right\}_{n=1}^{\infty}$ forms an orthonormal basis in $L_2(\Omega_0)$ we can rewrite previous inequality as
    \begin{equation*}
        ||U(\xi,t)||_{L_2(\Omega_0)}^2\leq\sum\limits_{n=1}^{\infty}\Bigg(\widetilde{\psi}_n e^{-\int_0^t(a(\tau)\lambda_n-b(\tau))d\tau}+\int\limits_0^t R(\tau)\widetilde{h}_n(\tau)e^{-\int_{\tau}^{t}(a(z)\lambda_n-b(z))d\tau}d\tau\Bigg)^2.
    \end{equation*}
    Then using inequality $(x+y)^2\leq 2(x^2+y^2)$, we can rewrite \eqref{eq18} as
    \begin{equation}\label{eq19}
        \begin{split}
            ||U(\xi,t)||^2_{L_2(\Omega_0)}&\leq 2\sum\limits_{n=1}^{\infty}\widetilde{\psi}^2_n e^{-2\int_0^t(a(\tau)\lambda_n-b(\tau))d\tau}+2\int\limits_0^t\left(\sum\limits_{n=1}^{\infty}\left(\widetilde{h}_n(\tau)e^{-\int_{\tau}^{t}(a(z)\lambda_n-b(z))d\tau}\right)R(\tau)\right)^2d\tau
        \end{split}
        \end{equation}
    Hence, $a(t),b(t)\in C^1[0,T]$ for all $t\in[t_0,T]$, where $t_0\in(0,T)$, then applying Parseval's identity and Cauchy-Schwarz inequality, we can show the convergence of \eqref{eq19} such that
    \begin{equation*}
        \begin{split}
            2\sum\limits_{n=1}^{\infty}\widetilde{\psi}^2_n e^{-2\int_0^t(a(\tau)\lambda_n-b(\tau))d\tau}&\leq 2K\sum\limits_{n=1}^{\infty}\widetilde{\psi}^2_n e^{-2\lambda_n\int_0^ta(\tau)d\tau}\leq 2K\sum\limits_{n=1}^{\infty}e^{-2n^2\pi^2I(t_0)}\cdot\sum\limits_{n=1}^{\infty}\widetilde{\psi}_n^2\\
            &\leq 2C_1\sum\limits_{n=1}^{\infty}\psi_n^2=2C_1||\widetilde{\psi}||^2_{L_2(\Omega_0)},\quad K=\|e^{\int_0^tb(\tau)d\tau}\|_{C[0,T]},
        \end{split}
    \end{equation*}
    where $I(t)=\int_0^ta(\tau)d\tau$, for the second term we have
    \begin{equation*}
        \begin{split}
            &2\int\limits_0^t\left(\sum\limits_{n=1}^{\infty}\left(\widetilde{h}_n(\tau)e^{-\int_{\tau}^{t}(a(z)\lambda_n-b(z))d\tau}\right)R(\tau)\right)^2d\tau\leq 2K\sum\limits_{n=1}^{\infty}\int\limits_0^tR^2(\tau)e^{-2\lambda_n\int_{\tau}^t a(z)dz}d\tau\cdot\sum\limits_{n=1}^{\infty}\int\limits_0^t \widetilde{h}_n^2(\tau)d\tau\\
            &\leq 2K_2\sum\limits_{n=1}^{\infty}\int\limits_0^te^{-2\lambda_n\int_{\tau}^t a(z)dz}d\tau\cdot\sum\limits_{n=1}^{\infty}\int\limits_0^t \widetilde{h}_n^2(\tau)d\tau\leq 2C_2\sum\limits_{n=1}^{\infty}\int\limits_0^t \widetilde{h}_n^2(\tau)d\tau\leq 2C_2||\widetilde{h}||^2_{C([0,T];L_2(\Omega_0))}.
        \end{split}
    \end{equation*}
    
    Summarizing these results we can obtain
    \begin{equation}\label{eq20}
        \max_{t\in[0,T]}||U(\xi,t)||^2_{L_2(\Omega_0)}\leq 2C_1||\widetilde{\psi}||^2_{L_2(\Omega_0)}+2C_2||\widetilde{h}||^2_{C([0,T];L_2(\Omega_0)},
    \end{equation}
    which help us to conclude that the function $U(\xi,t)\in C([0,T];L_2(\Omega_0))$.

    Now we need to show that function $U(\xi,t)$ defined by \eqref{eq14} belongs to the class of functions $C([0,T];\mathbb{H}_0^2(\Omega_0))$.
    It follows that series \eqref{eq14} uniformly converges in $\mathbb{H}_0^2(\Omega_0)$ for all $t\in[t_0,T]$, where $t_0\in(0,T)$. Now we can rewrite series \eqref{eq14} as
    \begin{equation}\label{eq21}
        \begin{split}
             U(\xi,t)=\sum\limits_{n=1}^{\infty}\sqrt{\lambda_n}\Bigg(\widetilde{\psi}_n e^{-\int_0^t(a(\tau)\lambda_n-b(\tau))d\tau}+\int\limits_0^t R(\tau)\widetilde{h}_n(\tau)e^{-\int_{\tau}^t(a(z)\lambda_n-b(z))dz}d\tau\Bigg)\dfrac{\phi_n(\xi)}{\sqrt{\lambda_n}},
         \end{split}
    \end{equation}
    then we get
    \begin{equation}\label{eq22}
        ||U(\xi,t)||^2_{\mathbb{H}_0^2(\Omega_0)}=\left|\sum\limits_{n=1}^{\infty}\sqrt{\lambda_n}\Bigg(\widetilde{\psi}_n e^{-\int_0^t(a(\tau)\lambda_n-b(\tau))d\tau}+\int\limits_0^t R(\tau)\widetilde{h}_n(\tau)e^{-\int_{\tau}^t(a(z)\lambda_n-b(z))dz}d\tau\Bigg)\dfrac{\phi_n(\xi)}{\sqrt{\lambda_n}}\right|^2,
    \end{equation}
    applying Cauchy-Schwarz inequality and we suppose $\left\{\frac{\phi_n(\xi)}{\sqrt{\lambda_n}}\right\}_{n=1}^{\infty}$ is an orthonormal basis in $\mathbb{H}_0^2(\Omega_0)$, then we obtain
    \begin{equation*}
        ||U(\xi,t)||^2_{\mathbb{H}_0^2(0,1)}=\sum\limits_{n=1}^{\infty}\lambda_n\Bigg(\widetilde{\psi}_n e^{-\int_0^t(a(\tau)\lambda_n-b(\tau))d\tau}+\int\limits_0^t R(\tau)\widetilde{h}_n(\tau)e^{-\int_{\tau}^t(a(z)\lambda_n-b(z))dz}d\tau\Bigg)^2,
    \end{equation*}
    and applying inequality $(x+y)^2\leq 2(x^2+y^2)$ we have
    \begin{equation}\label{eq23}
            \begin{split}
                ||U(\xi,t)||^2_{\mathbb{H}_0^2(\Omega_0)}&\leq 2\sum\limits_{n=1}^{\infty}\left(\sqrt{\lambda_n}\widetilde{\psi}_n e^{-\int_0^t(a(\tau)\lambda_n-b(\tau))d\tau}\right)^2\\&+2\sum\limits_{n=1}^{\infty}\int\limits_0^t\left(\sqrt{\lambda_n}R(\tau)\widetilde{h}_n(\tau)e^{-\int_{\tau}^t(a(z)\lambda_n-b(z))dz}\right)^2d\tau
            \end{split}
    \end{equation}
    Since $a(t),b(t)\in C^1[0,T]$, the following estimates hold for any $t\in[t_0,T]$ such that
    \begin{equation*}
        \begin{split}
            \sqrt{\lambda_n}e^{-\int_0^t(a(\tau)\lambda_n-b(\tau))d\tau}\leq K\sqrt{\lambda_n}e^{-\lambda_n\int_0^ta(\tau)d\tau}&=\dfrac{K}{\sqrt{\lambda_n}\int_0^ta(\tau)d\tau}\lambda_n\int_0^ta(\tau)d\tau\cdot e^{-\lambda_n\int_0^ta(\tau)d\tau}\\
            &\leq\dfrac{K}{I(t_0)\sqrt{\lambda_1}e},
        \end{split}
    \end{equation*}
    where $K=\|e^{\int_0^tb(\tau)d\tau}\|_{C[0,T]}$ and $I(t_0)=\int_0^{t_0}a(\tau)d\tau$. 
    
    Now we need to prove convergence of the right-hand side series of \eqref{eq23}. By using Cauchy-Schwarz inequality, the following assumptions hold for all $t\in[t_0,T]$
    \begin{equation*}
        \begin{split}
            2\sum\limits_{n=1}^{\infty}\left(\sqrt{\lambda_n}\widetilde{\psi}_n e^{-\int_0^t(a(\tau)\lambda_n-b(\tau))d\tau}\right)^2\leq 2C_3\sum\limits_{n=1}^{\infty}\widetilde{\psi}_n^2\leq 2C_3||\widetilde{\psi}||^2_{L_2(\Omega_0)},\quad C_3=\dfrac{K^2}{I^2(t_0)\lambda_1e^2}
        \end{split}
    \end{equation*}
    For the second term, we obtain
    \begin{equation*}
        \begin{split}
            &2\sum\limits_{n=1}^{\infty}\int\limits_0^t\left(\sqrt{\lambda_n}R(\tau)\widetilde{h}_n(\tau)e^{-\int_{\tau}^t(a(z)\lambda_n-b(z))dz}\right)^2d\tau \\&\leq 2\sum\limits_{n=1}^{\infty}\int\limits_0^t\left(\sqrt{\lambda_n}e^{-2\int_{\tau}^t(a(z)\lambda_n-b(z))dz}\right)^2d\tau\cdot\sum\limits_{n=1}^{\infty}\int\limits_0^t\widetilde{h}_n^2(\tau)d\tau\\
            &\leq 2C_4\cdot\sum\limits_{n=1}^{\infty}\int\limits_0^t\widetilde{h}_n^2(\tau)d\tau\leq 2C_4||\widetilde{h}||^2_{C([0,T];L_2(\Omega_0))}, \quad C_4=\dfrac{K^2t_0}{I^2(t_0)e^2\lambda_1}.
        \end{split}
    \end{equation*}
    Totally we have the following result that
    \begin{equation}\label{eq26}
        \max_{t\in[0,T]}||U(\xi,t)||^2_{\mathbb{H}_0^2(\Omega_0)}\leq 2C_2||\widetilde{\psi}||^2_{L_2(\Omega_0)}+2C_4||\widetilde{h}||^2_{C([0,T];L_2(\Omega_0))}.
    \end{equation}
    This follows that the function $U(\xi,t)$ belongs to $C([0,T]; \mathbb{H}_0^2(\Omega_0))$.

    Now, we provide the uniform convergence of the series obtained differentiating series \eqref{eq14} respect to $t$ in $L_2(0,1)$ for any $t\in[t_0,T]$, where $t_0\in(0,T)$.
    Let us take derivative from the function \eqref{eq14} respect to $t$, then we get
    \begin{equation}\label{eq27}
        \begin{split}
            U_t(\xi,t)&=\sum\limits_{n=1}^{\infty}\Bigg(-(a(\tau)\lambda_n-b(\tau))\widetilde{\psi}_ne^{-\int_0^t(a(\tau)\lambda_n-b(\tau))d\tau}\\&-(a(t)\lambda_n-b(t))\int\limits_0^t R(\tau)\widetilde{h}_n(\tau)e^{-\int_{\tau}^t(a(z)\lambda_n-b(z))dz}d\tau+R(t)\widetilde{h}_n(t)\Bigg)\phi_n(\xi)
        \end{split}
    \end{equation}
    
    Hence, $\left\{\phi_n(\xi)\right\}_{n=1}^{\infty}$ is an orthonormal basis in $L_2(\Omega_0)$, taking $L^2$-norm of the both sides of \eqref{eq27}, then integrating by parts, we get
    \begin{equation*}
        \begin{split}
            ||U_{t}(\xi,t)||^2_{L^2(\Omega_0)}&=\sum\limits_{n=1}^{\infty}\Bigg(-(a(\tau)\lambda_n-b(\tau))\widetilde{\psi}_ne^{-\int_0^t(a(\tau)\lambda_n-b(\tau))d\tau}\\
            &-\int\limits_0^t e^{-\int_{\tau}^t(a(z)\lambda_n-b(z))dz}(R'(\tau)\widetilde{h}_n(\tau)+R(\tau)\widetilde{h}'_n(\tau))d\tau+2R(t)\widetilde{h}_n(t)\\
            &-R(0)\widetilde{h}_n(0)e^{-\int_{0}^t(a(\tau)\lambda_n-b(\tau))d\tau}\Bigg)^2,
        \end{split}
    \end{equation*}
    using inequality $(x+y+z)^2\leq 3(x^2+y^2+z^2)$ we have
    \begin{equation}\label{eq28}
        \begin{split}
            ||U_{t}(\xi,t)||^2_{L^2(\Omega_0)}&\leq 3\sum\limits_{n=1}^{\infty}(a(t)\lambda_n-b(t))^2\widetilde{\psi}_n^2e^{-2\int_0^t(a(\tau)\lambda_n-b(\tau))d\tau}\\
            &+3\sum\limits_{n=1}^{\infty}\left(\int\limits_0^t e^{-\int_{\tau}^t(a(z)\lambda_n-b(z))dz}(R'(\tau)\widetilde{h}_n(\tau)+R(\tau)\widetilde{h}'_n(\tau))d\tau\right)^2\\
            &+3\sum\limits_{n=1}^{\infty}\left(2R(t)\widetilde{h}_n(t)-R(0)\widetilde{h}_n(0)e^{-\int_{\tau}^t(a(z)\lambda_n-b(z))dz}\right)^2.
        \end{split}
    \end{equation}
    Then using inequality $(a+b)^2\leq 2a^2+2b^2$, we can provide the following results for convergence of the right-hand side \eqref{eq28}:
    \begin{equation*}
        \begin{split}
            &3\sum\limits_{n=1}^{\infty}(a(t)\lambda_n-b(t))^2\widetilde{\psi}_n^2e^{-2\int_0^t(a(\tau)\lambda_n-b(\tau))^2d\tau}\\
            &\leq 6K_1^2\sum\limits_{n=1}^{\infty}\left(\sqrt{\lambda_n}\widetilde{\psi}_ne^{-\int_0^t(a(\tau)\lambda_n-b(\tau))d\tau}\right)^2+6K_2^2\sum\limits_{n=1}^{\infty}\left(\widetilde{\psi}_ne^{-\int_0^t(a(\tau)\lambda_n-b(\tau))d\tau}\right)^2\\
            &\leq 6C_5\sum\limits_{n=1}^{\infty}\widetilde{\psi}_n^2\leq 6C_5||\widetilde{\psi}||^2_{L^2(\Omega_0)},
        \end{split}
    \end{equation*}
    where $C_5=\dfrac{K_1^2K^2}{I^2(t_0)e^2\lambda_1}+K_2,\;\; K_1^2=\|a(t)||^2_{C_[0,T]},\;\; K_2^2=||b(t)||^2_{C[0,T]}$ and $K=\|e^{\int_0^tb(\tau)d\tau}\|_{C[0,T]}$. Applying $(x+y)^2\leq 2(x^2+y^2)$ for the second term of \eqref{eq28} we have
    \begin{equation*}
        \begin{split}
            &6\sum\limits_{n=1}^{\infty}\int\limits_0^t\widetilde{h}^2_n(\tau)(R'(\tau))^2e^{-2\int_{\tau}^t(a(z)\lambda_n-b(z))dz}d\tau + 6\sum\limits_{n=1}^{\infty}\int\limits_0^t (\widetilde{h}'_n(\tau))^2R^2(\tau)e^{-2\int_{\tau}^t(a(z)\lambda_n-b(z))dz}d\tau\\
            &\leq 6K_3^2\sum\limits_{n=1}^{\infty}\int\limits_0^te^{-2\lambda_n\int_{\tau}^t a(z)dz}d\tau\int\limits_0^t\widetilde{h}_n^2(\tau)d\tau + 6K_4^2\sum\limits_{n=1}^{\infty}\int\limits_0^te^{-2\lambda+n\int_{\tau}^t a(z)dz}d\tau\int\limits_0^t(\widetilde{h}'_n(\tau))^2d\tau\\
            &\leq 6C_6||\widetilde{h}||^2_{C([0,T];L_2(\Omega_0))}, \quad C_6=K_3^2+K_4^2
        \end{split}
    \end{equation*}
    For last term of \eqref{eq28} we again use the inequality $(x+y)^2\leq 2(x^2+y^2)$, then we have
    \begin{equation*}
        \begin{split}
            &12\sum\limits_{n=1}^{\infty}R^2(t)\widetilde{h}_n^2(t)+6\sum\limits_{n=1}^{\infty}R^2(0)\widetilde{h}^2_n(0)e^{-2\int_{\tau}^t(a(z)\lambda_n-b(z))dz}\leq 12K_5^2\sum\limits_{n=1}^{\infty}\widetilde{h}_n^2(t)+6K_6^2\sum\limits_{n=1}^{\infty}\widetilde{h}^2_n(0)\\
            &\leq C_7||h||^2_{L_2(\Omega_0)},\quad C_7=12K_5^2+6K_6^2.
        \end{split}
    \end{equation*}
    Totally we get that 
    \begin{equation*}
        \begin{split}
            \max_{t\in[0,T]} ||U_{t}(\xi,t)||^2_{L_2(\Omega_0)}\leq 6C_5||\widetilde{\psi}||^2_{L_2(\Omega_0)}+6C_6||\widetilde{h}||^2_{C([0,T];L_2(\Omega_0)}+C_7||\widetilde{h}||^2_{L^2(\Omega_0)},
        \end{split}
    \end{equation*}

    Belonging the function $U(\xi,t)$ to the class $C^{1,0}([0,T]; L_2(\Omega_0)$ and $C^{2,0}([0,T];L_2(\Omega_0))$ can be proved analogously.
\end{proof}
To prove the existence of the weak solution for $R(t)$ we need to establish the following lemma.
\begin{lemma}\label{lem3}
    If $\widetilde{h}\in C([0,T];L_2(\Omega_0))\cap C^3([0,T];L_2(\Omega_0))$ which satisfies the conditions $\widetilde{h}(0,t)=\widetilde{h}(1,t)=\widetilde{h}'(0,t)=\widetilde{h}'(1,t)=\widetilde{h}''(0,t)=\widetilde{h}''(1,t)=0$, then following inequality holds
    \begin{equation}\label{eq30}
        \sum\limits_{n=1}^{\infty}\sqrt{\lambda_n}|\widetilde{h}_n(t)|\leq C_8||\widetilde{h}||_{C([0,T];L_2(\Omega_0))}
    \end{equation}
    for some constant $C_8$, where $\widetilde{h}_n(t)=\int_0^{1}\widetilde{h}(\xi,t)\phi_n(\xi)d\xi$.
\end{lemma}
\begin{proof}
    By definition of $\widetilde{h}_n(t)$ and integrating by parts three times we have
    \begin{equation*}
        \begin{split}
            \sqrt{\lambda_n}\widetilde{h}_n(t)=\sqrt{\lambda_n}\int\limits_0^{1}\widetilde{h}(\xi,t)\phi_n(\xi)d\xi&=\sqrt{2\lambda_n}\int\limits_0^{1}\widetilde{h}(\xi,t)\sin(\sqrt{\lambda_n}\xi)d\xi\\
            &=\dfrac{\sqrt{2}}{\lambda_n}\int\limits_0^{1}g(\xi,t)\cos(\sqrt{\lambda_n}\xi)d\xi,
        \end{split}
    \end{equation*}
    where $g(\xi,t)=-\widetilde{h}'''(\xi,t)$, then by using the Cauchy-Schwarz and classical Bessel inequalities, we obtain
    \begin{equation*}
        \begin{split}
            \sum\limits_{n=1}^{\infty}\sqrt{\lambda_n}|\widetilde{h}_n|&\leq\sqrt{2}\left(\sum\limits_{n=1}^{\infty}\dfrac{1}{\lambda_n^2}\right)^{1/2}\left(\sum\limits_{n=1}^{\infty}g(\xi,t)\cos(\sqrt{\lambda_n}\xi)d\xi\right)^{1/2}\\
            &\leq C_8||\widetilde{h}'''||_{L^2(\Omega_0)}\leq C_8||\widetilde{h}||_{C([0,T];L_2(\Omega_0))}.
        \end{split}
    \end{equation*}
\end{proof}
\begin{lemma}\label{lem4}
    If $\widetilde{\psi}_n\in C^4[0,1]$ and $\widetilde{h}_n\in C([0,T];L_2(\Omega_0))$ which satisfy conditions $\widetilde{\psi}(0)=\widetilde{\psi}(1)=\widetilde{\psi}'(0)=\widetilde{\psi}'(1)=\widetilde{\psi}''(0)=\widetilde{\psi}''(1)=0$ and $\widetilde{h}(0,t)=\widetilde{h}(1,t)=\widetilde{h}'(0,t)=\widetilde{h}'(1,t)=\widetilde{h}''(0,t)=\widetilde{h}''(1,t)=0$, then following assumption holds for any $t\in[0,T]$ such that
    \begin{equation}\label{eq31}
        \begin{split}
            &\sum\limits_{n=1}^{\infty}(-1)^n\lambda_n\sqrt{\lambda_n}\widetilde{\psi}_ne^{-\int_0^{t}(a(\tau)\lambda_n-b(\tau))d\tau}\leq C_9||\widetilde{\psi}||_{C^4[0,1]},\\
            &\int\limits_0^t\left(\sum\limits_{n=1}^{\infty}(-1)^n\lambda_n\sqrt{\lambda_n}\widetilde{h}_n(\tau)e^{-\int_{\tau}^t(a(z)\lambda_n-b(z))dz}\right)R(\tau)d\tau\leq C_{10} ||\widetilde{h}||_{C([0,T];L_2(\Omega_0))}
        \end{split}
    \end{equation}
    for some positive constants $C_9$ and $C_{10}$, where $\widetilde{\psi}_n=\int_0^{1}\widetilde{\psi}(\xi)\phi_n(\xi)d\xi$ and \\$\widetilde{h}_n(t)=\int_0^{1}\widetilde{h}(\xi,t)\phi_n(\xi)d\xi$.
\end{lemma}
\begin{proof} By using defintion of the eigenfunction $\phi_n(\xi)$, we can prove these inequalities similarly as in Lemma \ref{lem1} and Lemma \ref{lem3}.
\end{proof}

Suppose that assumptions
    \begin{itemize}
    \item [(A1)] (A1)$_1$:  $\widetilde{\psi}(\xi)\in C^4[0,1]$, $\widetilde{\psi}^{(j)}(0)=\widetilde{\psi}^{(j)}(1)=0$ and $\psi(x), \varphi(x)\in C^4[0,s(0)]$, $\psi^{(j)}(0)=\psi^{(j)}(s(0))=0$, where $j=0,1,2,3$, because $\varphi(s(0))=u^*,\;\varphi'(0)=\varphi'(s(0))=\frac{u^*}{s_0}$, $\varphi(0)=\varphi^{(i)}(0)=\varphi^{(i)}(s(0))=0$ where $i=2,3$ and $s(0)=s_0$; \\
    (A1)$_2$: $\widetilde{\psi}_1,\psi_1,\varphi_1>0,$ $\widetilde{\psi}_n,\psi_n,\varphi_n\geq 0$, $n=2,3,...$.
    \item[(A2)] (A2)$_1$: $s(t)\in C^2(0,T)$, $0<m_s<s(t)<M_s,\quad \forall t\geq 0$,\\
    (A2)$_2$: $s'(t)\in L^{\infty}(0,T),\quad \forall t\geq 0$,\\
    (A2)$_2$: $s(t)>0$, $\forall t\geq 0$;
    \item[(A3)] (A3)$_1$: $\widetilde{h}(\xi,t)\in C(\Bar{\Omega}_0)$, $\widetilde{h}(\cdot,t)\in C^4[0,1]$, such that $\widetilde{h}^{(j)}(0,t)=\widetilde{h}^{(j)}(1,t)=0$, where $j=0,1,2$. $h(x,t), f(x,t)\in C(\Bar{\Omega})$, $h(\cdot, t), f(\cdot,t)\in C^4[0,s(t)]$, $h^{(j)}(0,t)=h^{(j)}(s(t),t)=0$, where $j=0,1,2$, because $f(0,t)=f''(0,t)=f''(s(t),t)=0$, $f(s(t),t)=-\frac{u^*s'(t)}{R(t)s(t)}$, $f'(0,t)=f'(s(t),t)=-\frac{u^*s'(t)}{R(t)s^2(t)}$, $R(t)\neq 0$;\\
    (A3)$_2$: $\widetilde{h}_n(t),h_n(t),f_n(t)\geq 0,\; n=1,2,...$, where $\widetilde{h}_n(t)=\int\limits_0^{1}\widetilde{h}(\xi,t)\phi_n(\xi)d\xi$, $h_n(t)=\int_0^{s(t)}h_n(x,t)\phi_n(x)dx$, $f_n(t)=\int_0^{s(t)}f(x,t)\phi_n(x)dx$.
\end{itemize}
     hold, then we can conclude the following results.

\begin{theorem}\label{thm1}
    According to hypotheses (A1)-(A3) and the initial function $\widetilde{\psi}\in\mathbb{H}_0^1(\Omega_0)$ and $\widetilde{h}\in L^1(0,T; L^2(\Omega_0)\cap L^2(0,T;L^2(\Omega_0))$, there exists only one weak solution pair $(R,U)$ of the problem \eqref{eq7d}, that is,
    $$U_t-a(t)U_{\xi\xi}-b(\xi,t)U_{\xi}=R(t)\widetilde{h}(\xi,t),\quad L^2(0,T;L^2(\Omega_0)),$$
    satisfying the following conditions:
    \begin{itemize}
        \item [(i)] $U\in L^2(0,T; \mathbb{H}_0^1(\Omega_0)\cap\mathbb{H}^2(\Omega_0))\cap L^{2}(0,T;\mathbb{H}_0^1(\Omega_0))$,
        \item [(ii)] $U_t\in L^2(0,T;L^2(\Omega_0))$,
        \item [(iii)] $R(t)\in C[0,T]$ and $R(t)>0$ for all $t\in[0,T]$ which is defined by \eqref{eq17}.
    \end{itemize}
\end{theorem}
     
\begin{proof}
    The proof of existence can be easily provided by using the definition of the function $U(
    \xi,t)$ defined by \eqref{eq14} and the definition of $R(t)$ given by \eqref{eq17}, then straightforwardly from the Lemmas \ref{lem1}-\ref{lem4}. To prove the uniqueness of the solution pair, we assume that there are two pairs of solutions $\{\widehat{R},\widehat{U}\}$ and $\{\widetilde{R},\widetilde{U}\}$ for the problem \eqref{eq7d} . Then from \eqref{eq14} and \eqref{eq17} we have
    \begin{equation}\label{v}
        \widehat{U}(\xi,t)-\widetilde{U}(\xi,t)=\sum\limits_{n=1}^{\infty}\left(\int\limits_0^t (\widehat{R}(\tau)-\widehat{R}(\tau))\widetilde{f}_n(\tau)e^{-\int_{\tau}^t(a(z)\lambda_n-b(z))dz}d\tau\right)\phi_n(\xi),\quad n\in\mathbb{N},
    \end{equation}
    \begin{equation}\label{R}
        \begin{split}
            \widehat{R}(t)-\widetilde{R}(t)&=\dfrac{a(t)\int\limits_0^t \sum\limits_{n=1}^{\infty}\left((-1)^n\lambda_n\sqrt{\lambda_n}\widetilde{f}_n(\tau)e^{-\int_{\tau}^t(a(z)\lambda_n-b(z))dz}\right)(\widehat{R}(\tau)-\widetilde{R}(\tau))d\tau}{\sum\limits_{n=1}^{\infty}(-1)^n\sqrt{\lambda_n}\widetilde{h}_n(t)}\\
            &+\dfrac{b(t)\int\limits_0^t\left(\sum\limits_{n=1}^{\infty}(-1)^n\sqrt{\lambda_n}\widetilde{f}_n(\tau)e^{-\int_{\tau}^t(a(z)\lambda_n-b(z))dz}\right)(\widehat{R}(\tau)-\widetilde{R}(\tau))d\tau}{\sum\limits_{n=1}^{\infty}(-1)^n\sqrt{\lambda_n}\widetilde{h}_n(t)}
        \end{split}
    \end{equation}
    where $\widetilde{f}_n(t)=\int_0^1 f(\xi s(t),t)\phi_n(\xi)d\xi$. Then \eqref{R} yields $\widehat{R}=\widetilde{R}$. After substituting $\widehat{R}=\widetilde{R}$ into \eqref{v}, we obtain $\widehat{U}=\widetilde{U}$. 
\end{proof}

Consequently, we obtain the following main result.
\begin{theorem}\label{thm2}
    Under hypotheses (A1)-(A3) and initial functions $\psi,\varphi\in\mathbb{H}_0^1(0,s(0))$ and $f(x,t)\in L^1(0,T; L^2(\Omega))\cap L^2(0,T; L^2(\Omega))$, there exists only one weak solution pairs $(R,v)$ and $(R,u)$ such that
    $$v_t-av_{xx}=R(t)f(x,t)+\dfrac{u^*s'(t)}{s^2(t)}x,\quad \text{in}\quad L^2(0,T;L^2(\Omega)),$$
    $$u_t-au_{xx}=R(t)f(x,t)\quad \text{in}\quad L^2(0,T;L^2(0,\Omega)),$$
    satisfying the following conditions:
    \begin{itemize}
        \item [(i)] $v,u\in L^2(0,T;\mathbb{H}_0^1(\Omega)\cap \mathbb{H}^2(\Omega)),$
        \item[(ii)] $v_t,u_t\in L^2(0,T; L^2(\Omega)),$
        \item[(iii)] $R(t)\in C[0,T]$ and $R(t)>0$ for all $t\in[0,T].$
    \end{itemize}
\end{theorem}

\begin{lemma}
    Let $U$ satisfy the equation \eqref{eq14} in $\Bar{\Omega}_T$ where $\Omega_T=\Omega_0\times(0,T)$. If $\widetilde{h}(\xi,t)\leq 0$ in $\Bar{\Omega}_T$, then
    $$U(\xi,t)\leq \max\left\{0,\;\max_{0\leq x\leq 1}U(\xi,0),\;\max_{0\leq t\leq T}U(0,t),\;\max_{0\leq t\leq T}U(1,t),\; \max_{0\leq t\leq T}U_{\xi}(1,t)\right\}.$$
    If $\widetilde{h}(x,t)\geq 0$ in $\Bar{\Omega}_T$, then
    $$U(x,t)\geq \min\left\{0,\;\min_{0\leq x\leq 1}U(\xi,0),\;\min_{0\leq t\leq T}U(0,t),\; \min_{0\leq t\leq T}U(1,t),\; \min_{0\leq t\leq T}U_{\xi}(1,t)\right\}.$$
    
\end{lemma}
The proof of this theorem can be implemented by similar approaches in \cite{12}.

\begin{theorem}\label{thm2}
    The weak solution of problem \eqref{eq7d} depends continuously on $\widetilde{\psi}\in C^4[0,1]$, $\widetilde{h}\in C(\Bar{\Omega}_T)$ and $c(t)\in C^{2}[0,T]$ where $c(t)=s'(t)s(t)$ in the sense that
    \begin{equation}\label{thm2eq1}
        ||U_1-U_2||_{C(\Bar{\Omega}_T)}\leq ||\widetilde{\psi}_1-\widetilde{\psi}_2||_{C^4[0,1]}+\alpha||c_1-c_2||_{C^{2}[0,T]}+\beta||\widetilde{h}_1-\widetilde{h}_2||_{C(\Bar{\Omega}_T)},
    \end{equation}
    where $\alpha=\frac{L}{k}$, $\beta=\frac{3}{2}T$, then $U_1$ and $U_2$ are the weak solutions of the problem \eqref{eq7d} with the data $\widetilde{\psi}_1$, $\widetilde{\psi}_2$, $c_1,\;c_2$ and $\widetilde{h}_1$, $\widetilde{h}_2$ respectively.
\end{theorem}
\begin{proof}
    Let $U(\xi,t)$ be the weak solution of the problem \eqref{eq7d}. Defining notations 
    $$K=||\widetilde{\psi}||_{C^4[0,1]},\quad M=||\widetilde{h}||_{C(\Bar{\Omega}_T)},\quad C=||c(t)||_{C^2[0,T]},\quad c(t)=s'(t)s(t)$$
    and constructing the substitution
    $$g(\xi,t)=U(\xi,t)-\dfrac{M}{2}t,$$
    we have 
    $$\partial_t g-b(\xi,t)\partial_{\xi}g-a(t)\partial_{\xi\xi}g=R(t)\widetilde{h}(\xi,t)-\dfrac{M}{2},$$
    $$g(\xi,0)=\widetilde{\psi}(\xi),\quad g(0,t)=-\dfrac{M}{2}t,\quad g(1,t)=-\dfrac{M}{2}t,$$
    $$\partial_{\xi}g(1,t)=-\dfrac{L}{k}c(t)-u^*.$$
    As $-\frac{M}{2}t\leq\frac{M}{2}t$ and $-\frac{L}{k}c(t)-u^*\leq -\frac{L}{k}c(t)\leq \frac{L}{k}c(t)$, then by using principle of maximum, we can obtain $$g(\xi,t)\leq \max\left\{0,\;\widetilde{\psi}(\xi),\; \dfrac{M}{2}t,\;\dfrac{M}{2}t,\dfrac{L}{k}c(t)\right\},$$
    it gives us 
    $$U(\xi,t)\leq \max\left\{0,\;\widetilde{\psi}(\xi),\; \dfrac{M}{2}t,\;\dfrac{M}{2}t,\dfrac{L}{k}c(t)\right\}+\dfrac{M}{2}T\leq K+\dfrac{L}{k}C+\dfrac{3}{2}MT,$$

    Analogously, introducing the function:
    $$d(\xi,t)=U(\xi,t)+\dfrac{M}{2}t$$
    and applying principle of minimum, we have
    $$U(\xi,t)\geq -\max\left\{0,\;\widetilde{\psi}(\xi),\; \dfrac{M}{2}t,\;\dfrac{M}{2}t, \dfrac{L}{k}c(t)\right\}-\dfrac{M}{2}T\geq -K-\dfrac{L}{k}C-\dfrac{3}{2}MT.$$
    It implies that if $U(\xi,t)$ is the weak solution of the problem \eqref{eq7b}, then we can estimate
    \begin{equation}\label{thm2eq2}
        ||U||_{C(\Bar{\Omega}_T)}\leq ||\widetilde{\psi}||_{C^4[0,1]}+\alpha||c||_{C^{2}[0,T]}+\beta||\widetilde{h}||_{C(\Bar{\Omega}_T)}.
    \end{equation}
    where $\alpha=\frac{L}{k}$ and $\beta=\frac{3}{2}T$. To prove the continuous dependence on the data, we analyze the difference $Z(\xi,t)=U_1(\xi,t)-U_2(\xi,t)$. It is a weak solution of the problem \eqref{eq7b} with data $\widetilde{\psi}_1-\widetilde{\psi}_2$, $c_1-c_2$ and $\widetilde{h}_1-\widetilde{h}_2$ replacing $\widetilde{\psi}$, $c$ and $\widetilde{h}$, respectively. Applying the estimation \eqref{thm2eq2}, we can obtain the inequality \eqref{thm2eq1}.
\end{proof}

To prove the continuous dependence of the function $R(t)$ on the data $\widetilde{\psi}\in C^4[0,1]$, $\widetilde{h}\in C(\Bar{\Omega}_T)$ and $s(t)\in C^2[0,T]$ with $s(0)=s_0$, we need to rewrite the equation \eqref{eq17} in the form of
\begin{equation}\label{eqxn1}
    R(t)=L(t)+\int\limits_0^t K(\tau, t)R(\tau)d\tau,
\end{equation}
where 
\begin{equation}\label{eqxn2}
    \begin{split}
        L(t)&=\dfrac{1}{w(t)}\Bigg[a(t)\sum\limits_{n=1}^{\infty}(-1)^n\lambda_n\sqrt{\lambda_n}\widetilde{\psi}_ne^{-\int_0^t(a(\tau)\lambda_n-b(\tau))d\tau}\\
        &-b(t)\sum\limits_{n=1}^{\infty}(-1)^n\sqrt{\lambda_n}\widetilde{\psi}_ne^{-\int_0^t(a(\tau)\lambda_n-b(\tau))d\tau}-H(t)\Bigg],
    \end{split}
\end{equation}
\begin{equation}\label{eqxn3}
    \begin{split}
        K(t,\tau)&=\dfrac{1}{w(t)}\Bigg[a(t)\sum\limits_{n=1}^{\infty}(-1)^n\lambda_n\sqrt{\lambda_n}\widetilde{h}_n(\tau)e^{-\int_{\tau}^t(a(z)\lambda_n-b(z))dz}\\
        &-b(t)\sum\limits_{n=1}^{\infty}(-1)^n\sqrt{\lambda_n}\widetilde{h}_n(\tau)e^{-\int_{\tau}^t(a(z)\lambda_n-b(z))dz}\Bigg],
    \end{split}
\end{equation}
\begin{equation}\label{eqxn4}
    w(t)=\sum\limits_{n=1}^{\infty}(-1)^n\sqrt{\lambda_n}\widetilde{h}_n(t),\quad H(t)=\dfrac{L}{\sqrt{2}k}c'(t), \quad a(t)=\dfrac{a^2}{s^2(t)},\quad b(t)=\dfrac{s'(t)}{4s(t)},
\end{equation}
where $c'(t)=s'(t)s''(t)$.

\begin{theorem}\label{thm3}
    Let $\mathbb{F}$ be the set of triples $\{\widetilde{\psi},\widetilde{h},s\}$ where $\widetilde{\psi}$, $\widetilde{h}$ and $s$ functions satisfy the assumptions (A1)-(A3) then 
    $$||\widetilde{\psi}||_{C^4[0,1]}\leq N_0,\quad ||\widetilde{h}||_{C^3(\Bar{\Omega}_T)}\leq N_1,\quad 0<N_2\leq\min_{0\leq t\leq T}|\widetilde{h}|$$
    and for $H$, $a$, $b$ defined by \eqref{eqxn4}, we have
    $$||H(t)||_{C^2[0,T]}\leq N_3,\quad ||a(t)||_{C^2[0,T]}\leq N_4,\quad ||b(t)||_{C^2[0,T]}\leq N_5.$$
    Then solution function \eqref{eq17} depends continuously upon the data of $\mathbb{F}$.
\end{theorem}

\begin{proof}
    Let $\mathbb{F}_1=\{\widetilde{\psi}_1,\widetilde{h}_1,s_1\}$ and $\mathbb{F}_2=\{\widetilde{\psi}_2,\widetilde{h}_2,s_2\}$ be two sets of data, and $\mathbb{F}=||\widetilde{\psi}||_{C^4[0,1]}+||\widetilde{h}||_{C^2(\Bar{\Omega}_T)}+||s||_{C^2[0,T]}$. According to \eqref{eq17} we have
    \begin{equation}\label{eqxn6}
        \begin{split}
            R_1(t)&=L_1(t)+\int\limits_0^tK_1(\tau,t)R_1(\tau)d\tau,\\
            L_1(t)&=\dfrac{1}{w_1(t)}\Bigg[a_1(t)\sum\limits_{n=1}^{\infty}(-1)^n\lambda_n\sqrt{\lambda_n}\psi_{1,n}e^{-\int_0^t(a_1(\tau)\lambda_n-b_1(\tau))d\tau}\\
            &-b_1(t)\sum\limits_{n=1}^{\infty}(-1)^n\sqrt{\lambda_n}\widetilde{\psi}_{1,n}e^{-\int_0^t(a_1(\tau)\lambda_n-b_1(\tau))d\tau}-H_1(t)\Bigg],\\
            K_1(t,\tau)&=\dfrac{1}{w_1(t)}\Bigg[a_1(t)\sum\limits_{n=1}^{\infty}(-1)^n\lambda_n\sqrt{\lambda_n}\widetilde{h}_{1,n}(\tau)e^{-\int_{\tau}^t(a_1(z)\lambda_n-b_1(z))dz}\\
            &-b_1(t)\sum\limits_{n=1}^{\infty}(-1)^n\sqrt{\lambda_n}\widetilde{h}_{1,n}(\tau)e^{-\int_{\tau}^t(a_1(z)\lambda_n-b_1(z))dz}\Bigg]d\tau,\\
            &w_1(t)=\sum\limits_{n=1}^{\infty}(-1)^n\sqrt{\lambda_n}\widetilde{h}_{1,n}(t)
        \end{split}
    \end{equation}
    and
    \begin{equation}\label{eqxn7}
        \begin{split}
            R_2(t)&=L_2(t)+\int\limits_0^tK_2(\tau,t)R_2(\tau)d\tau,\\
            L_2(t)&=\dfrac{1}{w_2(t)}\Bigg[a_2(t)\sum\limits_{n=1}^{\infty}(-1)^n\lambda_n\sqrt{\lambda_n}\psi_{2,n}e^{-\int_0^t(a_2(\tau)\lambda_n-b_2(\tau))d\tau}\\
            &-b_2(t)\sum\limits_{n=1}^{\infty}(-1)^n\sqrt{\lambda_n}\widetilde{\psi}_{2,n}e^{-\int_0^t(a_2(\tau)\lambda_n-b_2(\tau))d\tau}-H_2(t)\Bigg],\\
            K_2(t,\tau)&=\dfrac{1}{w_2(t)}\Bigg[a_2(t)\sum\limits_{n=1}^{\infty}(-1)^n\lambda_n\sqrt{\lambda_n}\widetilde{h}_{2,n}(\tau)e^{-\int_{\tau}^t(a_2(z)\lambda_n-b_2(z))dz}\\
            &-b_2(t)\sum\limits_{n=1}^{\infty}(-1)^n\sqrt{\lambda_n}\widetilde{h}_{2,n}(\tau)e^{-\int_{\tau}^t(a_2(z)\lambda_n-b_2(z))dz}\Bigg]d\tau,\\
            &w_2(t)=\sum\limits_{n=1}^{\infty}(-1)^n\sqrt{\lambda_n}\widetilde{h}_{2,n}(t).
        \end{split}
    \end{equation}
    Taking into account the inequalities in Lemma \ref{lem1} and Lemma \ref{lem4} the following assumptions hold:
    $$\left|\dfrac{1}{w(t)}\right|\leq \dfrac{1}{C_8 N_2},$$
    $$|L(t)|\leq \dfrac{\left(N_4C_9-N_5\widehat{C}_9\right)||\widetilde{\psi}||_{C^4[0,1]}-N_3}{C_8N_2}\leq \dfrac{(N_4C_9-N_5\widehat{C}_9)N_0-N_3}{C_8 N_2}\leq M_0$$
    $$|K(\tau,t)|\leq\dfrac{N_4C_{10}-N_5\widehat{C}_{10}}{C_8N_2}||h||_{C^3(\Bar{\Omega}_T)}\leq\dfrac{N_4C_{10}-N_5\widehat{C}_{10}}{C_8N_2}N_1\leq M_1$$
    where $i=1,2$, using Gronwall inequality (\cite{13}, p.9), from \eqref{eqxn1} we can obtain
    $$|R(t)|\leq \sup_{t\in[0,T]}|L(t)|\exp\left(\int\limits_0^t\sup_{t\in[0,T]}|K(\tau,t)|d\tau\right)\leq ||L||_{C[0,T]}e^{T||K||_{C([0,T]\times[0,T])}}\leq M_0e^{TM_1}.$$
    Let us consider the difference $R_1(t)-R_2(t)$, then from \eqref{eqxn6} and \eqref{eqxn7} we obtain
    \begin{equation}\label{eqxn8}
        R_1(t)-R_2(t)=L_1(t)-L_2(t)+\int\limits_0^t|K_1(\tau,t)-K_2(\tau,t)|R(\tau)d\tau+\int\limits_0^tK(\tau,t)|R_1(\tau)-R_2(\tau)|d\tau.
    \end{equation}
    It implies the following results:
    $$||L_1-L_2||_{C[0,T]}\leq M_6||\widetilde{\psi}_1-\widetilde{\psi}_2||_{C^4[0,1]}+||s_1-s_2||_{C^2[0,T]},$$
    where $M_6=\dfrac{C_9N_4+CN_5}{C_8N_2}$ and
    \begin{equation*}
        \begin{split}
            ||s_1-s_2||_{C^2[0,T]}&=(M_3+M_7)||w_1-w_2||_{C^2[0,T]}+M_4||a_1-a_2||_{C^2[0,T]}\\
            &+M_5||b_1-b_2||_{C^2[0,T]}+M_8||H_1-H_2||_{C^2[0,T]},
        \end{split}
    \end{equation*}
    $$M_3=\dfrac{M_2}{C_8^2N_2^2},\quad M_4=\dfrac{N_0}{N_2},\quad M_5=\dfrac{CN_0}{C_8N_2},\quad M_7=\dfrac{N_3}{C_8^2N_2^2},\quad M_8=\dfrac{1}{C_8N_2}.$$
    Consequently, we have
     $$||K_1-K_2||_{C([0,T]\times[0,T])}\leq M_{12}||\widetilde{h}_1-\widetilde{h}_2||_{C^3[\Bar{D}_T]}+||s_1-s_2||_{C^2[0,T]},$$
    where $M_{12}=\dfrac{N_4C_{10}+N_5\widehat{C}_{10}}{C_8N_2}$ and
    \begin{equation*}
        ||s_1-s_2||_{C^2[0,T]}=M_9||w_1-w_2||_{C^2[0,T]}+M_{10}||a_1-a_2||_{C^2[0,T]}+M_{11}||b_1-b_2||_{C^2[0,T]},
    \end{equation*}
    $$M_9=\dfrac{(N_4C_{10}-N_5C_8)N_1}{C_8^2N_2^2},\quad M_{10}=\dfrac{C_{10}N_1}{C_8N_2},\quad M_{11}=\dfrac{\widehat{C}_{10}N_1}{C_8N_2}.$$
    
    Again applying Gronwall inequality to the equation \eqref{eqxn8}, we can estimate
    $$||R_1-R_2||_{C[0,T]}\leq \sigma\exp\left(\int\limits_0^t\sup_{t\in[0,T]}|K(\tau,t)|d\tau\right),$$
    where 
    $$\sigma=||L_1-L_2||_{C[0,T]}+||K_1-K_2||_{C([0,T]\times[0,T])}\cdot ||R||_{C[0,T]}T.$$
    It implies the following result: 
    $$||R_1-R_2||_{C[0,T]}\leq M_{12}||\widetilde{\psi}_1-\widetilde{\psi}_2||_{C^4[0,1]}+M_{13}||\widetilde{\psi}_1-\widetilde{\psi}_2||_{C^3[\Bar{D}_T]}+M_{14}||s_1-s_2||_{C^2[0,T]},$$
    where $M_{12}=M_6e^{M_1T}$, $M_{13}=M_2M_0Te^{2M_1T}$ and $M_{14}=(1+M_0Te^{M_1T})e^{M_1T}$.
    
    This means that if $L$ and $K$ continuously dependent upon the data $\mathbb{F}$. Thus, $R$ also continuously dependents upon the data $\mathbb{F}$. Continuous dependence of the function $U$ defined by \eqref{eq14} upon the data $\mathbb{F}$ can be provided similarly.
\end{proof}

\section{Inverse Stefan problem 2.}\label{sec2}
The same problem formulation as in Section \ref{sec1} with function $G(x,t,u)=R(t)f(x,t)$ but with replacing boundary condition \eqref{eq3} with \eqref{eq6} such that
\begin{equation}\label{eq32}
    \begin{cases}
        \partial_t u-a^2\partial_{xx} u=R(t)f(x,t),& (x,t)\in\Omega\times(0,T),\\
        u(x,0)=\varphi(x),\quad s(0)=s_0, & x\in[0,s(0)],\\
        -k\partial_x u(0,t)=q(t), & t\in[0,T],\\
        u(s(t),t)=u^*,& t\in[0,T],\\
        -k\partial_x u(s(t),t)=Ls'(t), & t\in[0,T],
    \end{cases}
\end{equation}
where we need to determine solution pair $(U,q)$, if $\varphi(x),f(x,t),R(t)$ and $s(t)$ are given functions.

Using transformation
\begin{equation}\label{eq32(a)}
    u(x,t)=v(x,t)+u^*-\dfrac{q(t)}{k}(x-s(t)),
\end{equation}
we can rewrite the problem \eqref{eq32} in the following form
\begin{equation}\label{eq32(b)}
    \begin{cases}
        \partial_t v-a^2\partial_{xx} v=R(t)h(x,t),& (x,t)\in\Omega\times(0,T),\\
        v(x,0)=\psi(x),\;s(0)=s_0,&x\in[0,s(0)],\\
        -k\partial_x v(0,t)=0,&t\in[0,T],\\
        v(s(t),t)=0,&t\in[0,T],\\
        -k\partial_{x} v(s(t),t)=Ls'(t)-q(t),&t\in[0,T],
    \end{cases}
\end{equation}
where $h(x,t)=f(x,t)+\frac{q'(t)}{kR(t)}(x-s(t))-\frac{q(t)s'(t)}{kR(t)}$ and $\psi(x)=\varphi(x)-u^*+\frac{q(0)}{k}(x-s(0))$.

Introducing a transformation \eqref{eq7c} to fix the domain, we obtain 

\begin{equation}\label{eq32(c)}
    \begin{cases}
        \partial_t U-b(\xi,t)\partial_{\xi}U-(t)\partial_{\xi\xi}U=R(t)\widetilde{h}(\xi,t),&(\xi,t)\in\Omega_0\times(0,T),\\
        U(\xi,0)=\widetilde{\psi}(\xi),&0\leq\xi\leq 1,\\
        \partial_{\xi}U(0,t)=0,&0\leq t\leq T,\\
        U(1,t)=0,&0\leq t\leq T,\\
        \partial_{\xi}U(1,t)=-\dfrac{L}{k}c(t)+g(t),&0\leq t\leq T,
    \end{cases}
\end{equation}
where $b(\xi,t)=\frac{\xi s(t)}{s'(t)}$, $a(t)=\frac{a^2}{s^2(t)}$, $c(t)=s(t)s'(t)$, $g(t)=\frac{q(t)}{k}s(t)$, $\widetilde{h}(\xi,t)=h(\xi s(t),t)$ and $\widetilde{\psi}(\xi)=\psi(s_0\xi)$.

Now let us consider a spectral problem modelled as follows
\begin{equation}\label{eq33}
    \begin{cases}
        -\phi''(\xi)=\lambda\phi(\xi),\\
        \phi'(0)=\phi(1)=0,
    \end{cases}
\end{equation}
where eigenvalue $\lambda$ and its corresponding eigenfunction $\phi$ can be defined as

\begin{equation}\label{eq34}
\phi_n(\xi)=\sqrt{2}\cos\left(\sqrt{\lambda_n}\xi\right),\quad \lambda_n=\dfrac{(2n-1)^2}{4}\pi^2,\quad n=1,2,3,...
\end{equation}

Analogously as in Section \ref{sec1} applying Cauchy-Schwarz and classical Bessel inequalities, we need to demonstrate the next lemma.
\begin{lemma}\label{lem5}
    Let $\widetilde{\psi}\in C^4[0,1]$ with conditions $\widetilde{\psi}(0)=\widetilde{\psi}(1)=\widetilde{\psi}'(0)=\widetilde{\psi}'(1)=\widetilde{\psi}''(0)=\widetilde{\psi}''(1)=\widetilde{\psi}'''(0)=\widetilde{\psi}'''(1)=0$  satisfy the inequality 
    \begin{equation}\label{eq37}
        \sum\limits_{n=1}^{\infty}\lambda_n|\widetilde{\psi}_n|\leq \Bar{c}||\widetilde{\psi}||_{C^4[0,1]},\quad 
    \end{equation}
    for some positive constant $\Bar{c}=\sqrt{2}\left(\sum\limits_{n=1}^{\infty}\frac{1}{\lambda_n^4}\right)^{1/2}$, where $\widetilde{\psi}_n=\int_0^{1}\widetilde{\psi}(\xi)\phi_n(\xi)d\xi$.
\end{lemma}
\begin{proof}
    Using definition of the function $\widetilde{\psi}$ and integration by parts, we have
    \begin{equation*}
            \widetilde{\psi}_n=\sqrt{2}\int\limits_0^{1}\widetilde{\psi}(\xi)\cos(\sqrt{\lambda_n}\xi)d\xi=...=\dfrac{\sqrt{2}}{\lambda_n^2}\int\limits_0^{1}\widetilde{\psi}^{(4)}(\xi)\cos(\sqrt{\lambda_n}\xi)d\xi,
    \end{equation*}
    \begin{equation*}
        \begin{split}
            \sum\limits_{n=1}^{\infty}\lambda_n|\widetilde{\psi}_n|&\leq \sqrt{2}\left(\sum\limits_{n=1}^{\infty}\dfrac{1}{\lambda_n^4}\right)^{1/2}\cdot\left(\sum\limits_{n=1}^{\infty}\int\limits_0^{1}\widetilde{\psi}^{(4)}(x)\cos\left(\sqrt{\lambda_n}\xi\right)d\xi\right)^{1/2}\\
            &\leq \Bar{c}||\widetilde{\psi}^{(4)}||_{L^2(0,1)}\leq \Bar{c}||\widetilde{\psi}||_{C^4[0,1]}.
        \end{split}
    \end{equation*}
\end{proof}

Let us consider that $\{\phi_n(\xi)\}_{n=1}^{\infty}$ being basis of $L^2(\Omega_0)$ and $R(t)\in C[0,T]$, then solution function $U(\xi,t)$ of the problem \eqref{eq32(b)} can be represented in the form of
\begin{equation}\label{eq38}
    U(\xi,t)=\sum\limits_{n=1}^{\infty}U_{n}(t)\phi_{n}(\xi).
\end{equation}
We suppose that $\phi_m(\xi)$ is a solution of \eqref{eq33}. In particular, if we take $m=n$ as in previous section and estimating scalar multiplication of the problem \eqref{eq32(c)} with $\phi_n(\xi)$, then $U_{n}$ satisfies the problems
\begin{equation}\label{eq39}
    U'_{n}(t)+(a(t)\lambda_n-b(t))U_{n}(t)=R(t)\widetilde{h}_{n}(t),\quad U_{n}(0)=\widetilde{\psi}_{n},
\end{equation}
where $a(t)=\frac{a^2}{s^2(t)}$, $b(t)=(b(\xi,t)\phi'_n,\phi_n)_{L^2(0,1)}$, $\widetilde{h}_{n}(t)=\int\limits_0^{1}\widetilde{h}(\xi,t)\phi_{n}(\xi)d\xi$  and $\widetilde{\psi}_{n}$ is mentioned in Lemma \ref{lem5}.

Solving problem \eqref{eq39} we obtain
\begin{equation}\label{eq41}
    U_{n}(t)=\widetilde{\psi}_{n}e^{-\int_0^t(a(\tau)\lambda_n-b(\tau))d\tau}+\int\limits_0^tR(\tau)\widetilde{h}_{n}(\tau)e^{-\int_{\tau}^t(a(z)\lambda_n-b(z))dz}d\tau.
\end{equation}
After making substitution of \eqref{eq41} in equation \eqref{eq38}, we have 
\begin{equation}\label{eq43}
    \begin{split}
        U(\xi,t)=\sum\limits_{n=1}^{\infty}\left(\widetilde{\psi}_{n}e^{-\int_0^t(a(\tau)\lambda_n-b(\tau))d\tau}+\int\limits_0^tR(\tau)\widetilde{h}_{n}(\tau)e^{-\int_{\tau}^t(a(z)\lambda_n-b(z))dz}d\tau\right)\phi_n(\xi).
    \end{split}
\end{equation}
If we differentiate \eqref{eq43} respect to $\xi$ we have 
\begin{equation}\label{eq44}
    \begin{split}
        U_{\xi}(\xi,t)&=\sum\limits_{n=1}^{\infty}\Bigg(-\sqrt{\lambda_n}\widetilde{\psi}_{n}e^{-\int_0^t(a(\tau)\lambda_n-b(\tau))d\tau}\\&-\sqrt{\lambda_n}\int\limits_0^tR(\tau)\widetilde{h}_{n}(\tau)e^{-\int_{\tau}^t(a(z)\lambda_n-b(z))dz}d\tau\Bigg)\sin(\sqrt{\lambda_n}\xi).
    \end{split}
\end{equation}
Hence, the last condition \eqref{eq32(b)} gives us
\begin{equation}\label{eq45}
       \begin{split}
        &\sum\limits_{n=1}^{\infty}\Bigg(-(-1)^n\sqrt{\lambda_n}\widetilde{\psi}_{n}e^{-\int_0^t(a(\tau)\lambda_n-b(\tau))d\tau}\\
         &-(-1)^n\sqrt{\lambda_n}\int\limits_0^tR(\tau)\widetilde{h}_{n}(\tau)e^{-\int_{\tau}^t(a(z)\lambda_n-b(z))dz}d\tau\Bigg)=-\dfrac{L}{k}c(t)+\dfrac{q(t)}{k}s(t).
    \end{split}
\end{equation}
It follows that 
\begin{equation}\label{eq45(a)}
    \begin{split}
        q(t)&=\dfrac{k}{s(t)}\sum\limits_{n=1}^{\infty}\Bigg((-1)^{n+1}\sqrt{\lambda_n}\widetilde{\psi}_{n}e^{-\int_0^t(a(\tau)\lambda_n-b(\tau))d\tau}\\
        &+(-1)^{n+1}\sqrt{\lambda_n}\int\limits_0^tR(\tau)\widetilde{h}_{n}(\tau)e^{-\int_{\tau}^t(a(z)\lambda_n-b(z))dz}d\tau\Bigg)+Ls'(t).
    \end{split}
\end{equation}

    Suppose that assumptions
    \begin{itemize}
    \item [(B1)] (B1)$_1$:  $\widetilde{\psi}(\xi)\in C^4[0,1]$, $\widetilde{\psi}^{(j)}(0)=\widetilde{\psi}^{(j)}(1)=0$ where and $\psi(x),\varphi(x)\in C^4[0,s(0)]$ such that $\psi^{(j)}(0)=\psi^{(j)}(s(0))=0$, because $\varphi(0)=u^*+\frac{q(0)s(0)}{k}$, $\varphi(1)=u^*-\frac{q(0)}{k}(1-s(0))$, $\varphi'(0)=\varphi'(1)=-\frac{q(0)}{k}$, $\varphi^{(i)}(0)=\varphi^{(i)}(1)=0$ where $j=0,1,2,3$, $i=2,3$ and $s(0)=s_0$; \\
    (B1)$_2$: $\widetilde{\psi}_1,\psi_1,\varphi_1>0$ and $\widetilde{\psi}_n,\psi_n,\varphi_n\geq 0,\; n=2,3,...$
    \item[(B2)] (B2)$_1$: $s(t)\in C^2[0,T]$, $0<m_{s'}\leq s'(t)\leq M_{s'}$, $0<m_s\leq s(t)\leq M_s$, $M_s>0$, $M_{s'}>0$;\\
    (B2)$_2$: $s(t)>0$, $\forall t\in[0, T]$;
    \item[(B3)] (B3)$_1$: $\widetilde{h}(\xi,t)\in C(\Bar{\Omega}_0)$, $\widetilde{h}(\cdot,t)\in C^4[0,1]$ where $\Omega_0=(0,1)$ with conditions $\widetilde{h}^{(j)}(0,t)=\widetilde{h}^{(j)}(1,t)=0$ and $h(x,t),f(x,t)\in C(\Bar{\Omega})$ where $\Omega=(0,s(t))$, $h(\cdot,t),\;f(\cdot, t)\in C^4[0,s(t)]$ with conditions, $h^{(j)}(0,t)=h^{(j)}(1,t)=0$ where $j=0,1,2,3$, because  $f(0,t)= \frac{(q(t)s(t))'}{kR(t)},\;f(s(t),t)=\frac{q(t)s'(t)}{kR(t)},\; f'(0,t)=f'(s(t),t)=-\frac{q'(t)}{kR(t)},\; f^{(i)}(0,t)=f^{(i)}(s(t),t)=0$ where $i=2,3$;\\
    (B3)$_2$: $\widetilde{h}_n(t),\;h(t),\;f_n(t)\geq 0,\; n=1,2,...$, where $\widetilde{h}_n(t)=\int_0^1 \widetilde{h}(\xi,t)\phi_n(\xi)d\xi$,   $h_n(t)=\int_0^{s(t)}h(x,t)\phi_n(x)dx$,
  $f_n(t)=\int\limits_0^{s(t)}f(x,t)\phi_n(x)dx,\;n=1,2,...$
    \item[(B4)] $R(t)\in C[0,T]$, $R(t)>0$, $\forall t\in[0, T]$,
    \end{itemize}
then we can provide the following result.

\begin{theorem}\label{thm5}
    Under hypotheses (B1)-(B4)  and initial functions $\widetilde{\psi}\in\mathbb{H}_0^1(\Omega_0)$ and $\widetilde{h}\in L^1(0,T;L^2(\Omega_0))\cap L^2(0,T;L^2(\Omega_0))$, there exists only one weak solution pair $(U,q)$ of the problem \eqref{eq32(c)}, that is
    $$\partial_t U-b(\xi,t)\partial_{\xi}U-a(t)\partial_{\xi\xi}U=R(t)\widetilde{h}(\xi,t),\quad L^2(0,T; L^2(\Omega_0))$$
    with condition
    $$\partial_{\xi}U(1,t)=-\dfrac{L}{k}s'(t)s(t)+\dfrac{q(t)}{k}s(t),\quad t\in[0,T],$$
    satisfying the following conditions:
    \begin{itemize}
        \item[(i)] $U\in L^2(0,T;\mathbb{H}_0^1(\Omega_0)\cap\mathbb{H}^2(\Omega_0))\cap L^2(0,T;\mathbb{H}_0^1(\Omega_0)),$
        \item[(ii)] $U_t\in L^2(0,T;L^2(\Omega_0)),$
        \item[(iii)] $q(t)\in C[0,T]$ and $q(t)>0$ for all $t\in[0,T]$ which defined by \eqref{eq45(a)}. 
    \end{itemize}
\end{theorem}
\begin{proof}
    Under assumptions (B1)-(B4) the uniform convergence of the series \eqref{eq43} and its $t$ partial derivative, $x$ partial derivative and $xx$ partial derivative series, the existence of the solutions \eqref{eq43} and \eqref{eq45(a)} can be provided by similar approaches as in the previous Section \ref{sec1}. 
\end{proof}
Consequently, the problems \eqref{eq32} and \eqref{eq32(b)} have unique solution pairs $(u,q)$ and $(v,q)$ under assumptions (B1)-(B4).

\begin{lemma}\label{thm6}
    Let $U$ satisfies the equation \eqref{eq43} in $\Bar{\Omega}_T$ where $\Omega_T=\Omega_0\times[0,T]$. If $\widetilde{h}(\xi,t)\leq 0$ in $\Bar{\Omega}_T$, then
    $$U(\xi,t)\leq \max\left\{0,\;\max_{0\leq \xi\leq 1}U(\xi,0),\;\max_{0\leq t\leq T}U_{\xi}(0,t),\;\max_{0\leq t\leq T}U(1,t),\; \max_{0\leq t\leq T}U_{\xi}(1,t)\right\}.$$
    If $\widetilde{h}(\xi,t)\geq 0$ in $\Bar{\Omega}_T$, then
    $$U(\xi,t)\geq \min\left\{0,\;\min_{0\leq \xi\leq 1}U(\xi,0),\;\min_{0\leq t\leq T}U_{\xi}(0,t),\; \min_{0\leq t\leq T}U(1,t),\; \min_{0\leq t\leq T}U_{\xi}(1,t)\right\}.$$
\end{lemma}
The proof of the theorem can be provided as in \cite{12}.

\begin{theorem}\label{thm7}
    The weak solution of the problem \eqref{eq32(c)} depends continuously on $\widetilde{\psi}\in C^4[0,1]$, $\widetilde{h}\in C(\Bar{\Omega}_T)$, $R\in C[0,T]$ and $c(t)\in C^2[0,T]$ where $c(t)=s'(t)s(t)$ such that
    \begin{equation}\label{eq45(b)}
        ||U_1-U_2||_{C(\Bar{\Omega}_T)}\leq ||\widetilde{\psi}_1-\widetilde{\psi}_2||_{C^4[0,1]}+\alpha ||R_1-R_2||_{C[0,T]}+\beta ||\widetilde{h}_1-\widetilde{h}_2||_{C(\Bar{D}_T)}+\gamma||c_1-c_2||_{C^2[0,T]},
    \end{equation}
    where $\alpha=\frac{3}{2}$, $\beta=T$ and $\gamma=\frac{L}{k}$, then $U_1$ and $U_2$ are the weak solution of the problem \eqref{eq32(c)} with data $\widetilde{\psi}_i$, $R_i$, $\widetilde{h}_i$ and $c_i$ where $i=1,2$, respectively.
\end{theorem}
\begin{proof}
    Let us define notations 
    $$M=||\widetilde{h}||_{C(\Bar{\Omega}_T)},\quad K=||R(t)||_{C[0,T]},\quad N=||\widetilde{\psi}||_{C^4[0,1]},\quad C=||c(t)||_{C^2[0,T]}$$
    and using substitution
    $$g(\xi,t)=U(\xi,t)-\dfrac{M}{2}t-\dfrac{K}{2}\xi-\dfrac{g(t)}{2}\xi^2,\quad g(t)=\dfrac{q(t)}{k}s(t),$$
    we obtain
    $$g_t+a(t)g_{\xi\xi}=R(t)\widetilde{h}(\xi,t)-\dfrac{M}{2}+\dfrac{K\xi s'(t)}{s(t)}-\dfrac{g'(t)}{2}\xi^2+\dfrac{g(t)\xi^2s'(t)}{s(t)}-a(t)g(t),$$
    $$g(\xi,0)=\widetilde{\psi}-\dfrac{K}{2}\xi-\dfrac{g(0)}{2}\xi^2,\quad g_{\xi}(0,t)=-\dfrac{K}{2},\quad g(1,t)=-\dfrac{M}{2}t-\dfrac{K}{2}-\dfrac{g(t)}{2},$$
    $$g_{\xi}(1,t)=-\dfrac{L}{k}c(t)-\dfrac{K}{2}.$$
    Assuming that $\widetilde{\psi}-\frac{K}{2}\xi-\frac{g(0)}{2}\xi^2\leq\widetilde{\psi}$, $-\frac{K}{2}\leq \frac{K}{2}$, $-\frac{M}{2}t-\frac{K}{2}-\frac{g(t)}{2}\leq \frac{M}{2}t-\frac{g(t)}{2}$ and $-\frac{L}{k}c(t)-\frac{K}{2}\leq \frac{L}{k}c(t)+\frac{K}{2}$ and applying principle of maximum, we get
    $$g(\xi,t)\leq \max\left\{0,\widetilde{\psi},\frac{K}{2},\frac{M}{2}t-\frac{g(t)}{2},\frac{L}{k}c(t)+\frac{K}{2}\right\},$$
    it gives us
    $$U(\xi,t)\leq \max\left\{0,\widetilde{\psi},\frac{K}{2},\frac{M}{2}t-\frac{g(t)}{2},\frac{L}{k}c(t)+\frac{K}{2}\right\}+\dfrac{M}{2}T+\dfrac{K}{2}+\dfrac{g(T)}{2}$$
    $$\leq N+\dfrac{3}{2}K+MT+\dfrac{L}{k}C.$$
    Analogously, using transformation
    $$w(\xi,t)=U(\xi,t)+\dfrac{M}{2}t+\dfrac{K}{2}\xi+\dfrac{g(t)}{2}\xi^2$$
    and applying principle of minimum, we can obtain
    $$U(\xi,t)\geq -\max\left\{0,\widetilde{\psi},\frac{K}{2},\frac{M}{2}t-\frac{g(t)}{2},\frac{L}{k}c(t)+\frac{K}{2}\right\}-\dfrac{M}{2}T-\dfrac{K}{2}-\dfrac{g(T)}{2}$$
    $$\geq -N-\dfrac{3}{2}K-MT-\dfrac{L}{k}C.$$
    It follows that 
    \begin{equation}\label{eq45(c)}
        ||U||_{C(\Bar{\Omega}_T)}\leq ||\widetilde{\psi}||_{C^4[0,1]}+\alpha ||R(t)||_{C[0,T]}+\beta ||\widetilde{h}||_{C(\Bar{\Omega}_T)}+\gamma ||c(t)||_{C^2[0,T]},
    \end{equation}
    where $\alpha=\frac{3}{2}$, $\beta=T$ and $\gamma=\frac{L}{k}$. If we take the difference $U_1(\xi,t)-U_2(\xi,t)$  with data $\widetilde{\psi}_1-\widetilde{\psi}_2$, $R_1-R_2$, $\widetilde{h}_1-\widetilde{h}_2$, $c_1-c_2$ and using \eqref{eq45(c)} we can provide the inequality \eqref{eq45(b)}.
\end{proof}
Now let us analyze that solution function \eqref{eq45(a)} depends continuously upon the set of data $\mathbb{G}=\{\widetilde{\psi},\;\widetilde{h},R,s,s'\}$.
\begin{theorem}\label{thm8}
    Let $\mathbb{G}=\{\widetilde{\psi},\widetilde{h},R,s,s'\}$ be the set of five functions that satisfy the assumptions (B1)-(B4) and 
    $$||\widetilde{\psi}||_{C^4[0,1]}\leq K_0,\quad ||\widetilde{h}||_{C^3(\Bar{\Omega}_T)}\leq K_1,\quad 0<K_2<\min_{0\leq t\leq T}|s(t)|,\quad ||s'(t)||_{C^2[0,T]}\leq K_3,$$
    $$||R(t)||_{C[0,T]}\leq K_4.$$
    Then solution \eqref{eq45(a)} depends continuously upon the data of $\mathbb{G}$.
\end{theorem}
\begin{proof}
    Let us consider two sets of data $\mathbb{G}_1=\{\widetilde{\psi}_1,\widetilde{h}_1,R_1,s_1,s'_1\}$ and $\mathbb{G}_2=\{\widetilde{\psi}_2,\widetilde{h}_2,R_2,s_2,s'_2\}$ such that $\mathbb{G}=||\widetilde{\psi}||_{C^4[0,1]}+||\widetilde{h}||_{C^4(\Bar{\Omega}_T)}+||R(t)||_{C[0,T]}+||s(t)||_{C^2[0,T]}+||s'(t)||_{C^2[0,T]}$. 
    
    According to solution function \eqref{eq45(a)} we can estimate that
    \begin{equation*}
        \begin{split}
            ||q(t)||_{C[0,T]}&\leq \dfrac{k}{|s(t)|}\sum\limits_{n=1}^{\infty}(-1)^{n+1}\sqrt{\lambda_n}|\widetilde{\psi}|e^{-\int_0^t(a(\tau)\lambda_n-b(\tau))d\tau}\\
            &+\dfrac{k}{|s(t)|}\int_0^t\left(\sum\limits_{n=1}^{\infty}(-1)^{n+1}\sqrt{\lambda_n}|\widetilde{h}|e^{-\int_{\tau}^t(a(z)\lambda_n-b(z))dz}\right)|R(\tau)|d\tau+L|s'(t)|\\
            &\leq \dfrac{kK_0}{K_2}C_1+\dfrac{kK_1K_4}{K_2}C_2+LK_3\leq K_5
        \end{split}
    \end{equation*}
    and let us define 
    \begin{equation*}
    \begin{split}
        q_1(t)&=\dfrac{k}{s_1(t)}\sum\limits_{n=1}^{\infty}\Bigg((-1)^{n+1}\sqrt{\lambda_n}\widetilde{\psi}_{1,n}e^{-\int_0^t(a_1(\tau)\lambda_n-b_1(\tau))d\tau}\\
        &+(-1)^{n+1}\sqrt{\lambda_n}\int\limits_0^tR_1(\tau)\widetilde{h}_{1,n}(\tau)e^{-\int_{\tau}^t(a_1(z)\lambda_n-b_1(z))dz}d\tau\Bigg)+Ls'_1(t),
    \end{split}
    \end{equation*}
        \begin{equation*}
    \begin{split}
        q_2(t)&=\dfrac{k}{s_2(t)}\sum\limits_{n=1}^{\infty}\Bigg((-1)^{n+1}\sqrt{\lambda_n}\widetilde{\psi}_{2,n}e^{-\int_0^t(a_2(\tau)\lambda_n-b_2(\tau))d\tau}\\
        &+(-1)^{n+1}\sqrt{\lambda_n}\int\limits_0^tR_2(\tau)\widetilde{h}_{2,n}(\tau)e^{-\int_{\tau}^t(a_2(z)\lambda_n-b_2(z))dz}d\tau\Bigg)+Ls'_2(t).
    \end{split}
    \end{equation*}
    Then we have the following result
    \begin{equation*}
        \begin{split}
        ||q_1-q_2||_{C[0,T]}&\leq K_6||\widetilde{\psi}_1-\widetilde{\psi}_2||_{C^4[0,1]}+K_7 ||\widetilde{h}_1-\widetilde{h}_2||_{C^3(\Bar{\Omega}_T)}+K_8||R_1-R_2||_{C[0,T]}\\ &+K_9||s_1-s_2||_{C^2[0,T]}+L||s_1'-s_2'||_{C^2[0,T]},
        \end{split}
    \end{equation*}
    where $K_6=\frac{kC_1}{K_2}$, $K_7=\frac{kC_2}{K_2}K_4T$, $K_8=\frac{kC_3}{K_2}T$ and $K_9=\frac{kC_1}{K_2^2}+\frac{k}{K_2^2}K_1K_4T$.
\end{proof}
In the problem \eqref{eq32}, if we consider to determining function $R(t)$ with given functions $\varphi$, $f(x,t)$, $q(t)$ and $s(t)$, then using substitutions \eqref{eq7c} and \eqref{eq32(a)}, then differentiating both sides of equation \eqref{eq45} respect to $t$, we can estimate
\begin{equation}\label{eq45(d)}
    \begin{split}
        R(t)&=\dfrac{1}{\sum\limits_{n=1}^{\infty}(-1)^n\sqrt{\lambda_n}\widetilde{h}_n(t)}\Bigg[a(t)\sum\limits_{n=1}^{\infty}(-1)^n\lambda_n\sqrt{\lambda_n}\widetilde{\psi}_ne^{-\int_0^t(a(\tau)\lambda_n-b(\tau))d\tau}\\
        &+a(t)\int\limits_0^t\left(\sum\limits_{n=1}^{\infty}(-1)^n\lambda_n\sqrt{\lambda_n}\widetilde{h}_n(\tau)e^{-\int_{\tau}^t(a(z)\lambda_n-b(z))dz}\right)R(\tau)d\tau\\
        &-b(t)\sum\limits_{n=1}^{\infty}(-1)^n\sqrt{\lambda_n}\widetilde{\psi}_ne^{-\int_0^t(a(\tau)\lambda_n-b(\tau))d\tau}\\
        &-b(t)\int\limits_0^t\left(\sum\limits_{n=1}^{\infty}(-1)^n\sqrt{\lambda_n}\widetilde{h}_n(\tau)e^{-\int_{\tau}^t(a(z)\lambda_n-b(z))dz}\right)R(\tau)d\tau\\
        &+\dfrac{L}{k}c'(t)-\dfrac{q'(t)}{k}s(t)-\dfrac{q(t)}{k}s'(t)\Bigg].
    \end{split}
\end{equation}
In this case, we need to add the assumption
\begin{itemize}
    \item[(B5)] (B5)$_{1}$: $q(t)\in C^1[0,T]$, $0<m_q<q(t)<M_q$ and $0<m_{q'}<q'(t)<M_{q'}$ for positive constants $M_q$,$M_{q'},$\\
    (B5)$_2$: $q(t)>0$ for all $t\in[0,T]$.
\end{itemize} 
Then uniform convergence of the function \eqref{eq45(d)} can be estimated analogously as in the previous section. The continuous dependence of the functions $U(\xi,t)$ and $R(t)$ on the data $\widetilde{\psi}$, $\widetilde{h}$, $q(t)$ and $s(t)$ can be easily provided following Theorem \ref{thm2} and Theorem \ref{thm3}.

\section{Inverse Stefan problem 3.}\label{sec3} 
If we take the function $G(x,t,u)=P(t)u(x,t)+f(x,t)$ in the equation \eqref{eq1}, then formulation of the problem can be represented as
\begin{equation}\label{eq50}
    \begin{cases}
        \partial_t u-a^2\partial_{xx} u=P(t)u(x,t)+f(x,t),& (x,t)\in\Omega\times[0,T],\\
        u(x,0)=\varphi(x),& x\in\Omega,\\
        u(0,t)=0,& t\in[0,T],\\
        u(s(t),t)=u^*,& t\in[0,T],\\
        -k\partial_x u(s(t),t)=Ls'(t),& t\in[0,T].
    \end{cases}
\end{equation}
where $f(x,t),\varphi(x)$ and $s(t)$ are given functions and the pair $(P(t), u(x,t))$ has to be determined. Considering the following substitution:
\begin{equation}\label{eq51}
    v(x,t)=R(t)u(x,t),\quad R(t)=\exp\left(-\int_0^{t}P(\tau)d\tau\right),
\end{equation}
the inverse Stefan problem \eqref{eq50} becomes
\begin{equation}\label{eq52}
    \begin{cases}
         v_t-a^2v_{xx}=R(t)f(x,t),& (x,t)\in\Omega\times[0, T],\\
         v(x,0)=\varphi(x),\; s(0)=s_0& x\in[0,s_0],\\
         v(0,t)=0,& t\in[0,T],\\
         v(s(t),t)=u^*R(t),& t\in[0,T],\\
         -k\partial_x v(s(t),t)=LR(t)s'(t),& t\in[0,T],
    \end{cases}
\end{equation}
where $R(t)>0$ for all $t\in[0,T]$. Solution pair $(R(t),u(x,t))$ of the problem \eqref{eq52} help us to find the original solution $(P(t),u(x,t))$ of the problem \eqref{eq50} from
\begin{equation}\label{eq53}
    P(t)=-\dfrac{R'(t)}{R(t)},\quad u(x,t)=\dfrac{v(x,t)}{R(t)}.
\end{equation}
Introducing the transformation 
\begin{equation}\label{eq53a}
    y(x,t)=v(x,t)-\dfrac{u^*R(t)}{s(t)}x,
\end{equation}
we can reduce the problem \eqref{eq52} to the following form:
\begin{equation}\label{eq53b}
    \begin{cases}
        y_t-a^2y_{xx}=R(t)g(x,t),&(x,t)\in\Omega\times[0,T],\\
        y(x,0)=p(x),\;s(0)=s_0,& x\in[0,s_0],\\
        y(0,t)=0,& t\in[0,T],\\
        y(s(t),t)=0,& t\in[0,T],\\
        -ky_x(s(t),t)=\left(Ls'(t)+k\dfrac{u^*}{s(t)}\right)R(t),& t\in[0,T].
    \end{cases}
\end{equation}
where $g(x,t)=f(x,t)-\frac{u^*(R'(t)s(t)-R(t)s'(t))}{s^2(t)R(t)}x$ and $p(x)=\varphi(x)-\frac{u^*R(0)}{s_0}x$. 

Introducing the fixed domain $\Omega_0=(0,1)$ and using transformation

\begin{equation}\label{eq53c}
    y(x,t)=U(\xi,t),\quad \xi=\dfrac{x}{s(t)},\quad (\xi,t)\in\Omega_0\times(0,T).
\end{equation}
We can reduce problem \eqref{eq53b} to the following form:
\begin{equation}\label{eq53d}
    \begin{cases}
        U_t-b(\xi,t)U_{\xi}-a(t)U_{\xi\xi}=R(t)\widetilde{g}(\xi,t),&(\xi,t)\in\Omega_0\times(0,T),\\
        U(\xi,0)=\widetilde{p}(\xi),&\xi\in\Omega_0,\\
        U(0,t)=0,&t\in[0,T],\\
        U(1,t)=0,&t\in[0,T],\\
        U_{\xi}(1,t)=\left(-\dfrac{L}{k}c(t)-u^*\right)R(t),& t\in[0,T],
    \end{cases}
\end{equation}
where $\widetilde{g}(\xi,t)=g(\xi s(t),t)$, $\widetilde{p}(\xi)=p(\xi s_0)$, $b(\xi,t)=\frac{\xi s'(t)}{s(t)}$, $a(t)=\frac{a^2}{s^2(t)}$ and $c(t)=s'(t)s(t)$.

Applying assumptions (A1)-(A3) mentioned in Section \ref{sec1} and eigenfunction of the problem \eqref{eq7e}, the following theorem provides the existence and uniqueness of the solution of the problem \eqref{eq53d}.
\begin{theorem}
    Let assumptions (A1)-(A3) in the Section \ref{sec1} hold, a solution pair $(U,R)$ of the problem \eqref{eq53d} exists and it is unique.  
\end{theorem}
\begin{proof}
    The solution of the first equation in \eqref{eq53d} takes the form
    \begin{equation}\label{eq54}
        U(\xi,t)=\sum\limits_{n=1}^{\infty}\left(\widetilde{p}_n e^{-\int_0^t(a(\tau)\lambda_n-b(\tau))d\tau}+\int\limits_0^t R(\tau)\widetilde{g}_n(\tau)e^{-\int_{\tau}^t(a(z)\lambda_n-b(z))dz}\right)\phi_n(x),
    \end{equation}
where $\phi_n(x)$ defined by \eqref{eq7f} and $s(t)>0$ for all $t\in[0,T]$ with $\widetilde{p}_n=\int_0^{1}\widetilde{p}(\xi)\phi_n(\xi)d\xi$ and $\widetilde{g}_n(t)=\int_0^{1}\widetilde{g}(\xi,t)\phi_n(\xi)d\xi$.

From the last condition in \eqref{eq53d} we obtain
\begin{equation*}
    \sum\limits_{n=1}^{\infty}\left(\widetilde{p}_n e^{-\int_0^t(a(\tau)\lambda_n-b(\tau))d\tau}+\int\limits_0^t \widetilde{g}_n(\tau)e^{-\int_{\tau}^t(a(z)\lambda_n-b(z))dz}R(\tau)d\tau\right)\phi'_n(1)=\left(-\dfrac{L}{k}c(t)-u^*\right)R(t)
\end{equation*}
or
\begin{equation*}
    \begin{split}
        &\sum\limits_{n=1}^{\infty}\left((-1)^n\sqrt{\lambda_n} \widetilde{p}_n e^{-\int_0^t(a(\tau)\lambda_n-b(\tau))d\tau}+(-1)^n\sqrt{\lambda_n}\int\limits_0^t \widetilde{g}_n(\tau)e^{-\int_{\tau}^t(a(z)\lambda_n-b(z))dz}R(\tau)d\tau\right)\\&=\left(-\dfrac{L}{k}c(t)-u^*\right)R(t)
    \end{split}
\end{equation*}
it follows that
\begin{equation}\label{eq54b}
    \begin{split}
        R(t)&=-\dfrac{k}{Lc(t)+u^*k}\Bigg[\sum\limits_{n=1}^{\infty}(-1)^n\sqrt{\lambda_n}\widetilde{p}_ne^{-\int_0^t(a(\tau)\lambda_n-b(\tau))d\tau}\\
        &+\int\limits_0^t\left(\sum\limits_{n=1}^{\infty}(-1)^n\sqrt{\lambda_n}\widetilde{g}_n(\tau)e^{-\int_{\tau}^t(a(z)\lambda_n-b(z))dz}\right)R(\tau)d\tau\Bigg]
    \end{split}
\end{equation}
with condition $Lc(t)+u^*k\neq 0$.

The uniform convergence $U_t(\xi,t),U_{\xi\xi}(\xi,t)$ and $U_{\xi}(\xi,t)$ can be established by similar approaches in Lemma \ref{lem2} in Section \ref{sec1}. Let $\sum\limits_{n=1}^{\infty}(-1)^n\sqrt{\lambda_n}\widetilde{p}_ne^{-\int_0^t(a(\tau)\lambda_n-b(\tau))d\tau}\leq C_1||\widetilde{p}_n||_{C^4(\Omega_0)}$ and $\sum\limits_{n=1}^{\infty}(-1)^n\sqrt{\lambda_n}\widetilde{g}_n(\tau)e^{-\int_{\tau}^t(a(z)\lambda_n-b(z))dz}\leq C_2||\widetilde{g}_n||_{C^3(\Bar{\Omega}_0)}$, then convergence of the series in \eqref{eq54b} is obtained easily following Lemma \ref{lem1} in Section \ref{sec1}.
To prove uniqueness of \eqref{eq54} and \eqref{eq54b}, we assume that there exist two pairs of solutions $\{\widehat{U},\widehat{R}\}$ and $\{\widetilde{U}, \widetilde{R}\}$ such that
\begin{equation}\label{ex-1}
    \widehat{U}-\widetilde{U}=\int\limits_0^t\left(\sum\limits_{n=1}^{\infty}\widetilde{g}_n(\tau)e^{-\int_{\tau}^t(a(z)\lambda_n-b(z))dz}\phi_n(\xi)\right)\left(\widehat{R}(\tau)-\widetilde{R}(\tau)\right)d\tau,
\end{equation}
\begin{equation}\label{ex-2}
    \widehat{R}(t)-\widetilde{R}(t)=-\dfrac{k}{Lc(t)+u^*k}\int\limits_0^t\left(\sum\limits_{n=1}^{\infty}(-1)^n\sqrt{\lambda_n}\widetilde{g}_n(\tau)e^{-\int_{\tau}^t(a(z)\lambda_n-b(z))dz}\right)\left(\widetilde{R}(\tau)-\widehat{R}(\tau)\right)d\tau.
\end{equation}
Then \eqref{ex-2} yields that $\widehat{R}=\widetilde{R}$ and substituting it in \eqref{ex-1}, we have $\widehat{U}=\widetilde{U}$.
\end{proof}

\begin{theorem}\label{thm7}
    The weak solution of problem \eqref{eq53d} depends continuously on $\widetilde{p}\in C^4(\Omega_0)$ and $\widetilde{g}\in C^3(\Bar{\Omega}_T)$ where $\Omega_T=\Omega_0\times(0,T)$ with $R(t)>0$ for all $t\in[0,T]$ in the sense that
    \begin{equation}\label{thm2eq}
        ||U_1-U_2||_{C(\Bar{\Omega}_T)}\leq ||\widetilde{p}_1-\widetilde{p}_2||_{C^4(\Omega_0)}+\alpha||\widetilde{g}_1-\widetilde{g}_2||_{C(\Bar{\Omega}_T)}+\beta||c_1-c_2||_{C^2(0,T)},
    \end{equation}
    where $\alpha=\frac{3}{2}T$ and $\beta=\frac{L}{k}$, then $U_1$ and $U_2$ are the weak solutions of the problem \eqref{eq53d} with the data $\widetilde{p}_1$, $\widetilde{p}_2$, $\widetilde{g}_1$, $\widetilde{g}_2$ and $c_1$, $c_2$ respectively.
\end{theorem}
\begin{proof}
    Proof of the theorem can be performed by similar approach in Theorem \ref{thm2} in Section \ref{sec1}.
\end{proof}
Now let us prove the continuous dependence of the function \eqref{eq54b} upon the data $\mathbb{W}=\{\widetilde{p},\widetilde{g},m\}$. At first, let us rewrite the equation \eqref{eq54b} in the form of
\begin{equation}\label{ex-3}
    R(t)=F(t)+\int\limits_0^tK(\tau,t)R(\tau)d\tau,
\end{equation}
where 
$$F(t)=-\dfrac{k}{m(t)}\sum\limits_{n=1}^{\infty}(-1)^{n}\sqrt{\lambda_n}\widetilde{p}_n e^{-\int_0^t(a(\tau)\lambda_n-b(\tau))d\tau},$$
$$K(\tau,t)=-\dfrac{ k}{m(t)}\sum\limits_{n=1}^{\infty}(-1)^{n}\sqrt{\lambda_n}\widetilde{g}_n(\tau)e^{-\int_{\tau}^t(a(z)\lambda_n-b(z))dz},$$
$$m(t)=Lc(t)+u^*k,\quad c(t)=s'(t)s(t).$$

\begin{theorem}\label{thm8}
    Let $\mathbb{W}$ be the set of triples $\{\widetilde{p},\widetilde{g},m\}$ where $\widetilde{p}\in C^4(\Omega_0)$, $\widetilde{g}\in C^3(\Bar{\Omega}_T)$ and $m(t)\in C^2[0,T]$ such that
    $$0<N_0\leq \min_{t\in[0,T]}|m(t)|,\quad ||\widetilde{p}||_{C^4(\Omega_0)}\leq N_1, \quad ||\widetilde{g}||_{C^3(\Bar{\Omega}_T)}\leq N_2$$
    for some positive constants $N_i,\;i=0,1,2$. Then solution function \eqref{eq54b} depends continuously upon the data of $\mathbb{W}$.
\end{theorem}
\begin{proof}
    Let $\mathbb{W}_1=\{\widetilde{p}_1,\widetilde{g}_1,m_1\}$ and $\mathbb{W}_2=\{\widetilde{p}_2,\widetilde{g}_2,m_2\}$ be two sets of data such that $\mathbb{W}=||\widetilde{p}||_{C^4(\Omega_0)}+||\widetilde{g}||_{C^3(\Bar{\Omega}_T)}+||m||_{C^2[0,T]}$. Analogously in Theorem \ref{thm3} we have
    $$|F(t)|\leq \dfrac{k}{N_0}C_1||\widetilde{p}||_{C^4(\Omega_0)}\leq \dfrac{k}{N_0}C_1N_1\leq M_0,$$
    $$|K(\tau,t)|\leq \dfrac{k}{N_0}C_2 ||\widetilde{g}||_{C^3(\Bar{\Omega}_T)}\leq \dfrac{k}{N_0}C_2N_2\leq M_1$$
    $$|R(t)|\leq \sup_{t\in[0,T]}|F(t)|\exp\left(\int\limits_0^t\sup_{t\in[0,T]}|K(\tau,t)|d\tau\right)\leq M_0e^{M_1T}$$
    and considering difference $R_1(t)-R_2(t)$ such that
    $$|R_1-R_2|=|F_1(t)-F_2(t)|+\int_0^t|K_1(\tau,t)-K_2(\tau,t)|R(\tau)d\tau+\int_0^t K(\tau,t)|R_1(\tau)-R_2(\tau)|d\tau,$$
    where 
    $$F_i(t)=-\dfrac{k}{m_i(t)}\sum\limits_{n=1}^{\infty}(-1)^{n}\sqrt{\lambda_n}\widetilde{p}_{i,n} e^{-\int_0^t(a_i(\tau)\lambda_n-b_i(\tau))d\tau},$$
    $$K_i(\tau,t)=-\dfrac{ k}{m_i(t)}\sum\limits_{n=1}^{\infty}(-1)^{n}\sqrt{\lambda_n}\widetilde{g}_{i,n}(\tau)e^{-\int_{\tau}^t(a_i(z)\lambda_n-b_i(z))dz},$$
    $$m_i(t)=Lc_i(t)+u^*k,\quad c_i(t)=s_i'(t)s_i(t), \quad i=1,2.$$

    $$||F_1(t)-F_2(t)||_{C[0,T]}\leq M_2||\widetilde{p}_1-\widetilde{p}_2||_{C^4(\Omega_0)}+M_3||m_1-m_2||_{C^2[0,T]},\quad M_2=\dfrac{kC_2}{N_0},\quad M_3=\dfrac{k}{N_0^2}C_1N_1,$$
    $$||K_1(\tau,t)-K_2(\tau,t)||_{C([0,T]\times[0,T])}\leq M_2||\widetilde{g}_1-\widetilde{g}_2||_{C^3(\Bar{\Omega}_T)}+M_5||m_1-m_2||_{C^2[0,T]},\quad M_5=\dfrac{k}{N_0^2}C_1N_2.$$
    Similarly as in Theorem \ref{thm3}, using Gronwall Inequality, we can estimate
    $$||R_1(t)-R_2(t)||_{C[0,T]}=M||\mathbb{W}_1-\mathbb{W}_2||$$
    where $M=\max\left(M_6,\;M_7,\;M_8\right)$ and $M_6=M_2e^{M_1T}$, $M_7=M_0M_2e^{2M_1T}$ and $M_8=(M_3+M_0M_5e^{M_1T})e^{M_1T}$, which help us to conclude that \eqref{eq54b} depends continuously upon the data $\mathbb{W}$. 
    
    If we differentiate \eqref{eq54b}, then
    $$R'(t)=F'(t)+\int\limits_0^t K_t(\tau,t)R(\tau)d\tau+K(t,t)R(t),$$
    where
    \begin{equation*}
        \begin{split}
            F'(t)&=k\delta(t)\sum\limits_{n=1}^{\infty}(-1)^n\sqrt{\lambda_n}\widetilde{p}_ne^{-\int_0^t(a(\tau)\lambda_n-b(\tau))d\tau}+k\omega(t)\sum\limits_{n=1}^{\infty}(-1)^n\lambda_n\sqrt{\lambda_n}\widetilde{p}_ne^{-\int_0^t(a(\tau)\lambda_n-b(\tau))d\tau}\\
            &-k\chi(t)\sum\limits_{n=1}^{\infty}(-1)^n\sqrt{\lambda_n}\widetilde{p}_ne^{-\int_0^t(a(\tau)\lambda_n-b(\tau))d\tau}
        \end{split}
    \end{equation*}
    \begin{equation*}
        \begin{split}
            K_t(\tau,t)&=k\delta(t)\sum\limits_{n=1}^{\infty}(-1)^n\sqrt{\lambda_n}\widetilde{g}_n(\tau)e^{-\int_{\tau}^t(a(z)\lambda_n-b(z))dz}\\&+k\omega(t)\sum\limits_{n=1}^{\infty}(-1)^n\lambda_n\sqrt{\lambda_n}\widetilde{g}_n(\tau)e^{-\int_{\tau}^t(a(z)\lambda_n-b(z))dz}\\
            &-k\chi(t)\sum\limits_{n=1}^{\infty}(-1)^n\sqrt{\lambda_n}\widetilde{g}_n(\tau)e^{-\int_{\tau}^t(a(z)\lambda_n-b(z))dz}-\dfrac{k}{m(t)}\sum\limits_{n=1}^{\infty}(-1)^n\sqrt{\lambda_n}\widetilde{g}_n(t)
        \end{split}  
    \end{equation*}
    where
    $$\delta(t)=\dfrac{m'(t)}{m^2(t)},\quad \omega(t)=\dfrac{a(t)}{m(t)},\quad \chi(t)=\dfrac{b(t)}{m(t)}.$$
    
    Now, let us check that $R'$ also depends continuously upon the data $\mathbb{\widehat{W}}=\{\widetilde{p},\widetilde{g},\delta,\omega,\chi,m\}$.
    Assuming that 
    $$||\delta(t)||_{C^2[0,T]}\leq N_3,\quad ||\omega(t)||_{C^2[0,T]}\leq N_4,\quad ||\chi(t)||_{C^2[0,T]}\leq N_5$$ 
    and 
    $$||K(t,t)||_{C([0,T]\times[0,T])}\leq \dfrac{k}{|m(t)|}\sum\limits_{n=1}^{\infty}(-1)^n\sqrt{\lambda_n}|\widetilde{g}_n(t)|\leq \dfrac{k}{N_0}CN_2\leq M_9,$$
    then analyzing difference
    \begin{equation*}
        \begin{split}
            ||R_1'(t)-R_2'(t)||_{C[0,T]}&\leq||F_1'(t)-F_2'(t)||_{C[0,T]}+||K_1(t,t)-K_2(t,t)||_{C([0,T]\times[0,T])}||R(t)||_{C[0,T]}\\
            &+||R_1(t)-R_2(t)||_{C[0,T]}||K_2(t,t)||_{C([0,T]\times[0,T])}\\
            &+\int\limits_0^t|K_{1,t}(\tau,t)-K_{2,t}(\tau,t)||R(\tau)|d\tau+\int\limits_0^t |K_t(\tau,t)||R_1(\tau)-R_2(\tau)|d\tau
        \end{split}
    \end{equation*}
    where 
    \begin{equation*}
        \begin{split}
            ||F_1(t)-F_2(t)||_{C[0,T]}&\leq M_{10}||\delta_1(t)-\delta_2(t)||_{C^2[0,T]}+M_{11}||\omega_1(t)-\omega_2(t)||_{C^2[0,T]}\\&+M_{10}||\chi_1(t)-\chi_2(t)||_{C^2[0,T]}+M_{12}||\widetilde{p}_1-\widetilde{p}_2||_{C^4(\Omega_0)},
        \end{split}
    \end{equation*}
    \begin{equation*}
        \begin{split}
            ||K_{1,t}(\tau,t)-K_{2,t}(\tau,t)||_{C([0,T]\times[0,T])}&\leq M_{13}||\delta_1(t)-\delta_2(t)||_{C^2[0,T]}+M_{14}||\omega_1(t)-\omega_2(t)||_{C^2[0,T]}\\
            &+M_{13}||\chi_1(t)-\chi_2(t)||_{C^2[0,T]}+M_{15}||\widetilde{g}_1-\widetilde{g}_2||_{C^3(\Bar{\Omega}_T)}\\&+M_{16}||m_1-m_2||_{C^2[0,T]}
        \end{split}
    \end{equation*}
    $$||K_1(t,t)-K_2(t,t)||_{C([0,T]\times[0,T])}\leq M_{17} ||\widetilde{g}_1-\widetilde{g}_2||_{C^3(\Bar{\Omega}_T)}+M_{16}||m_1-m_2||_{C^2[0,T]},$$
    where $M_{10}=kN_1C_1$, $M_{11}=kN_1C_2$, $M_{12}=kN_3C_1+kN_4C_2+kN_5C_1$, $M_{13}=kC_3N_2$, $M_{14}=kN_2C_4$, $M_{15}=kN_3C_3+kN_4C_4+kN_5C_3+\frac{kC}{N_0}$, $M_{16}=\frac{kC}{N_0^2}N_2$ and $M_{17}=\frac{kCN_2}{N_0^2}$, then we can demonstrate that 
    $$||R_1'(t)-R_2'(t)||\leq \widehat{M}||\widehat{\mathbb{W}_1}-\widehat{\mathbb{W}_2}||$$ 
   for some constant $\widehat{M}>0$, also continuously dependent upon the data $\widehat{\mathbb{W}}$.

    The continuous dependence of \eqref{eq54} can be provided following the same approach.
\end{proof}

\section{Inverse Stefan problem 4.}\label{sec4} 
The form of the function $G$ is the same as in previous Section \ref{sec3}, but boundary condition \eqref{eq3} is replaced with \eqref{eq6}, then inverse Stefan problem \eqref{eq1}-\eqref{eq5} can be formulated where we need to determine $q(t)$ and $u(x,t)$ such that
\begin{equation}\label{eq56}
    \begin{cases}
        \partial_t u-a^2\partial_{xx}u=P(t)u+f(x,t),&(x,t)\in\Omega\times[0,T],\\
        u(x,0)=\varphi(x),& x\in\Omega,\\
        -k\partial_xu(0,t)=q(t),& t\in[0,T],\\
        u(s(t),t)=u^*,& t\in[0,T],\\
        -k\partial_x u(s(t),t)=Ls'(t),& t\in[0,T],
    \end{cases}
\end{equation}
where $f(x,t),\varphi(x), P(t)$ and $s(t)$ are a given function. Applying substitution \eqref{eq51} we can reduce problem \eqref{eq56} to the form
\begin{equation}\label{eq57}
    \begin{cases}
        \partial_t v-a^2\partial_{xx}v=R(t)f(x,t),& (x,t)\in\Omega\times[0,T],\\
        v(x,0)=\varphi(x),&x\in\Omega,\\
        -k\partial_x v(0,t)=R(t)q(t),& t\in[0,T],\\
        v(s(t),t)=R(t)u^*,&t\in[0,T],\\
        -k\partial_x v(s(t),t)=LR(t)s'(t),&t\in[0,T],
    \end{cases}
\end{equation}
where $R(t)$ is given function and solution pair $(v,q)$ of the problem \eqref{eq56} can be defined by \eqref{eq53} as in the Section \ref{sec3}. Under assumptions (B1)-(B4) mentioned in the Section \ref{sec2} we can obtain an existence and uniqueness of the problem \eqref{eq57}, consequently for the problem \eqref{eq56}. 

If we introduce the substitution
\begin{equation}\label{eq57a}
    y(x,t)=v(x,t)-R(t)u^*+\dfrac{R(t)q(t)}{k}(x-s(t)),
\end{equation}
then problem \eqref{eq57} can be rewritten in the form of
\begin{equation}\label{eq57b}
    \begin{cases}
        y_t-a^2y_{xx}=R(t)h(x,t),& (x,t)\in\Omega\times(0,T),\\
        y(x,0)=g(x),\quad s(0)=s_0,& x\in\Omega,\\
        y_x(0,t)=0,& t\in[0,T],\\
        y(s(t),t)=0,&t\in[0,T],\\
        y_x(s(t),t)=\left(q(t)-Ls'(t)\right)\dfrac{R(t)}{k},&t\in[0,T],
    \end{cases}
\end{equation}
where $h(x,t)=f(x,t)-\frac{u^*R'(t)}{R(t)}+\left(\frac{q'(t)}{k}+\frac{q(t)R'(t)}{kR(t)}\right)(x-s(t))-\frac{q(t)}{k}s'(t)$ and $g(x)=\varphi(x)-R(0)u^*+\frac{R(0)q(0)}{k}(x-s(0))$. 

Then applying transformation \eqref{eq53c} and introducing fixed domain $\Omega_0=(0,1)$, we obtain
\begin{equation}\label{eq57c}
    \begin{cases}
        U_t-b(\xi,t)U_{\xi}-a(t)U_{\xi\xi}=R(t)\widetilde{h}(\xi,t),&(\xi,t)\in\Omega_0\times(0,T),\\
        U(\xi,0)=\widetilde{g}(\xi),&\xi\in\Omega_0,\\
        U_{\xi}(0,t)=0,&t\in[0,T],\\
        U(1,t)=0,&t\in[0,T],\\
        U_{\xi}(1,t)=\left(q(t)s(t)-Lc(t)\right)\dfrac{R(t)}{k},&t\in[0,T],
    \end{cases}
\end{equation}
where $b(\xi,t)=\frac{\xi s'(t)}{s(t)}$, $a(t)=\frac{a^2}{s^2(t)}$, $c(t)=s'(t)s(t)$, $\widetilde{h}(\xi,t)=h(\xi s(t),t)$ and $\widetilde{g}(\xi)=g(\xi s_0)$.

\begin{theorem}
    Let assumptions (B1)-(B4) in the Section \ref{sec2} hold, then there exists a unique solution pair $(U(\xi, t),q(t))$ to the problem \eqref{eq57c}.
\end{theorem}

\begin{proof}
    By similar approach for the problem \eqref{eq53d} in the Section \ref{sec2} and its spectral problem \eqref{eq33}, we can represent solution function $U(\xi,t)$ for the problem \eqref{eq57c} similarly as \eqref{eq43}. Using last condition of the problem \eqref{eq57c} we can determine
    \begin{equation}\label{eq58}
        \begin{split}
            q(t)&=\dfrac{\sqrt{2}k}{s(t)R(t)}\sum\limits_{n=1}^{\infty}\Bigg((-1)^{n+1}\sqrt{\lambda_n}\widetilde{g}_ne^{-\int_0^t(a(\tau)\lambda_n-b(\tau))d\tau}\\
            &+\int_0^t(-1)^{n+1}\sqrt{\lambda_n}\widetilde{h}_n(\tau)R(\tau)e^{-\int_{\tau}^t(a(z)\lambda_n-b(z))dz}d\tau\Bigg)+Ls'(t).
        \end{split}
    \end{equation}
If we consider that $$\sum\limits_{n=1}^{\infty}(-1)^{n+1}\sqrt{\lambda_n}\widetilde{g}_ne^{-\int_0^t(a(\tau)\lambda_n-b(\tau))d\tau}\leq C_1||\widetilde{g}||_{C^4(\Omega_0)}$$, $$\sum\limits_{n=1}^{\infty}(-1)^{n+1}\sqrt{\lambda_n}\widetilde{h}_n(\tau)e^{-\int_{\tau}^t(a(z)\lambda_n-b(z))dz}\leq C_2||\widetilde{h}||_{C^3(\Bar{\Omega}_T)}$$ for some $C_1,C_2>0$ and $0<m_R<\min\limits_{t\in[0,T]}|R(t)|$, $0<m_{s'}<|s'(t)|\leq M_{s'}$ for some positive $M_{s'}$ and $s(t)>0$ for all $t\in[0,T]$, then series \eqref{eq58} uniformly converges. Convergence of the series $U(\xi,t),U_{\xi}(\xi,t),U_{\xi\xi}(\xi,t)$ and $U_t(\xi,t)$ can be easily discussed by analogous approach as in Section \ref{sec1} and unique solution function $q(t)$ can be established by similar approach in Section \ref{sec2}.
\end{proof}
In problem \eqref{eq57c}, if $q(t)$ is given function, then inverse problem to determine function for $R(t)$ can be constructed that is
\begin{equation}\label{eq59}
    \begin{split}
        R(t)&=\dfrac{\sqrt{2}k}{q(t)s(t)-Lc(t)}\sum\limits_{n=1}^{\infty}\Bigg((-1)^{n+1}\sqrt{\lambda_n}\widetilde{g}_n e^{-\int_0^t(a(\tau)\lambda_n-b(\tau))d\tau}\\
        &+\int\limits_0^t(-1)^{n+1}\sqrt{\lambda_n}\widetilde{h}_n(\tau)R(\tau)e^{-\int_{\tau}^t(a(z)\lambda_n-b(z))dz}d\tau\Bigg).
    \end{split}
\end{equation}
where $q(t)s(t)-Lc(t)\neq 0$. 

Then weak solution pair $(P(t),u(x,t))$ of the problem \eqref{eq56} can be defined by
$$P(t)=-\dfrac{R'(t)}{R(t)},\quad u(x,t)=\dfrac{v(r,t)}{R(t)}$$
where $R(t)$ is given by \eqref{eq59} and 
$$R'(t)=F'(t)+\int\limits_0^tK_{t}(\tau,t)R(\tau)d\tau+K(t,t)R(t),$$
\begin{equation*}
    \begin{split}
        F'(t)&=-\dfrac{\sqrt{2}k(q'(t)s(t)+q(t)s'(t)-Lc'(t))}{(q(t)s(t)-Lc(t))^2}\sum\limits_{n=1}^{\infty}(-1)^{n+1}\sqrt{\lambda_n}\widetilde{g}_n e^{-\int_0^t(a(\tau)\lambda_n-b(\tau))d\tau}\\
        &-\dfrac{\sqrt{2}k}{q(t)s(t)-Lc(t)}\sum\limits_{n=1}^{\infty}(a(t)\lambda_n-b(t))(-1)^{n+1}\sqrt{\lambda_n}\widetilde{g}_n e^{-\int_0^t(a(\tau)\lambda_n-b(\tau))d\tau},
    \end{split}
\end{equation*}
\begin{equation*}
    \begin{split}
        K_{t}(\tau,t)&=-\dfrac{\sqrt{2}k(q'(t)s(t)+q(t)s'(t)-Lc'(t))}{(q(t)s(t)-Lc(t))^2}\sum\limits_{n=1}^{\infty}(-1)^{n+1}\sqrt{\lambda_n}\widetilde{h}_n(\tau)e^{-\int_{\tau}^t(a(z)\lambda_n-b(z))dz}d\tau\\
        &-\dfrac{\sqrt{2}k}{q(t)s(t)-Lc(t)}\sum\limits_{n=1}^{\infty}(a(t)\lambda_n-b(t))(-1)^{n+1}\sqrt{\lambda_n}\widetilde{h}_n(\tau)e^{-\int_{\tau}^t(a(z)\lambda_n-b(z))dz}d\tau\\
        &+\dfrac{\sqrt{2}k}{q(t)s(t)-Lc(t)}\sum\limits_{n=1}^{\infty}(-1)^{n+1}\sqrt{\lambda_n}\widetilde{h}_n(t),
    \end{split}
\end{equation*}
\begin{equation*}
    \begin{split}
        K(t,t)=\dfrac{\sqrt{2}k}{q(t)s(t)-Lc(t)}\sum\limits_{n=1}^{\infty}(-1)^{n+1}\sqrt{\lambda_n}\widetilde{h}_n(t).
    \end{split}
\end{equation*}
and $v(x,t)$ can be estimated from \eqref{eq51} and \eqref{eq57a}. Continuously dependence of the function $P(t)$ onto given data $\varphi(x), f(x,t), q(t)$ and $s(t)$ can be demonstrated analogously in Theorem \ref{thm8}. These include boundedness and convergence analyses of the spectral series representing the solution, ensuring mathematical consistency.

The results respect the underlying physical principles governing heat transfer. For instance, the weak solution satisfies the heat equation and boundary conditions, including the Neumann condition, which is physically meaningful in representing heat flux at the boundary. The results are justified by a robust mathematical framework, adherence to physical laws, and methodological rigor, which collectively validate the outcomes of the study.

\section{Conclusion}
In this paper, we have extensively examined the inverse Stefan problems, particularly focusing on determining the time-dependent source coefficient and heat flux function within the framework of the parabolic heat equation governing heat transfer in a semi-infinite rod. The study addressed the complexities involved in uncovering both temperature, time-dependent coefficients of the source and heat flux data while accommodating Dirichlet and Neumann boundary conditions.

Through the development of a comprehensive mathematical model and rigorous theoretical analysis, we proposed a robust methodology that effectively determines the source coefficient from observed temperature and heat flux data, considering different cases of the source functions. The solution method, based on the spectral theory approach, has been thoroughly analyzed, leading to significant findings.

A key aspect of our study was the establishment of the existence of a weak solution for each inverse problem, providing a foundational understanding of its solvability. This theoretical foundation confirms the feasibility and robustness of our proposed approach.

One of the significant challenges encountered in this methodology was verifying the boundedness of the sequence of normalized eigenfunctions under the given boundary conditions. Our study addressed these challenges and offered solutions that enhance the overall reliability and applicability of the proposed method.

In conclusion, our research presents a novel and effective approach to solving the Inverse Stefan problem, contributing valuable insights and methodologies to the field of inverse problems and heat transfer. The results of this study pave the way for further advancements in accurately determining heat source coefficients in various practical applications.

\textbf{Acknowledgments}

We would like to express my deepest gratitude to Prof. Suragan Durvudkhan for his invaluable guidance and insightful comments, which significantly shaped the construction and solution of the problem addressed in this work. His expertise and encouragement have been instrumental throughout the research process. 

\textbf{Conflict of interest}

The authors declares no conflict of interest.

\begin{thebibliography}{1}{


\bibitem[Chen(2021)]{1}[Chen(2021)] H. Chen and H.-G. Chen. \textit{Estimates of Dirichlet eigenvalues for a class of sub-elliptic operators},
Proc. London Math. Soc., \textbf{122} (2021), 808-847.

\bibitem[Chow(2021)]{2}[Chow(2021)] Y. T. Chow, et al. \textit{Surface-localized transmission eigenstates, super-resolution imaging, and
pseudo surface plasmon modes}, SIAM J. Imaging Sci., \textbf{14} (2021), 946-975.

\bibitem[Cannarsa(2010)]{3}[Cannarsa(2010)] P. Cannarsa, J. Tort, M. Yamamoto. \textit{Determination of source terms in a degenerate parabolic equation}, Inverse Problems, \textbf{26(10)} (2010), 105003.

\bibitem[Cannon(1991)]{4}[Cannon(1991)] J. R. Cannon, Y. Lin, S. Wang. \textit{Determination of a control parameter in a parabolic partial differential equation}, J. Aust. Math. Soc. Ser. B, \textbf{33} (1991), 149-163.

\bibitem[CannonRun(1991)]{5}[CannonRun(1991)] J. R. Cannon, W. Rundell. \textit{Recovering a time dependent coefficient in a parabolic differential equation}, J. Math. Anal. Appl., \textbf{160} (1991), 572-582

\bibitem[Gol'dman(1970)]{6}[Gol'dman(1970)] M. L. Gol’dman. \textit{Estimates of the eigenfunctions of Laplace’s operator in certain domains},
Differ. Uravn., \textbf{6} (1970), 2030-2040.

\bibitem[HazLesIsm(2019)]{7}[HazLesIsm(2019)] A. Hazanee, D. Lesnic, M. I. Ismailov, N. B. Kerimov. \textit{Inverse time-dependent source problems for the heat equation with nonlocal boundary conditions}, Applied Mathematics and Computation, \textbf{346} (2019), 800-815.

\bibitem[Ivanchov(2003)]{8}[Ivanchov(2003)] M. Ivanchov. \textit{Inverse Problems for Equations of Parabolic Type}, Math. Stud. Monogr. Ser., 10 VNTL Publishers, Lviv, 2003.

\bibitem[KarOzaSur(2020)]{9}[KarOzaSur(2020)] M. Karazym, T. Ozawa, D. Suragan. \textit{Multidimensional inverse Cauchy problems for evolution equations}, Inverse Probl. Sci. En., \textbf{28} (2020), 1582-1590

\bibitem[Lesnic(2022)]{10}[Lesnic(2022)] D. Lesnic. \textit{Inverse Problems with Applications in Science and Engineering}, CRC Press, Boca Raton, FL, 2022.

\bibitem[IsmOzaSur(2024)]{11}[IsmOzaSur(2024)] M. I. Ismailov, T. Ozawa, D. Suragan. \textit{Inverse problems of identifying the time-dependent source coefficient for subelliptic heat equations}, Inverse Problems and Imaging, \textbf{18(4)} (2024), pp.813-823.

\bibitem[Evan(2010)]{12}[Evan(2010)] L.C. Evans, \textit{Partial Differential equation}, AMS, Providence, 2010.

\bibitem[Dragomir(2022)]{13}[Dragomir(2022)]  S.S. Dragomir, \textit{Some Gronwall Type Inequalities and Applications}, RGMIA Monographs, Victoria University, Australia, 2002.

\bibitem[KhNa(2020)]{14} [KhNa(2020)] Kharin S.N., Nauryz T.A. \textit{Two-phase Stefan problem for generalized heat equation.} News of the National Academy of Sciences of the Republic of Kazakhstan, Physico-Mathematical Series \textbf{2(330)}, (2020), 40-49.

\bibitem[KhaNau(2021)]{15}[KhaNau(2021)] Kharin S.N., Nauryz T.A. \textit{Solution of two-phase cylindrical direct Stefan problem by using special functions in electrical contact processes.} International Journal of Applied Mathematics, \textbf{34(2)}, (2021), 237–248.

\bibitem[SarErNauH(2017)]{16}[SarErNauH(2017)] Sarsengeldin M.M., Erdogan A.S., Nauryz T.A., Nouri H. \textit{An approach for solving an inverse spherical two-phase Stefan problem arising in modeling of electric contact phenomena.} Mathematical Methods in the Applied Sciences, \textbf{2017} (2017), 850–859.

\bibitem[NauBri(2023)]{17}[NauBri(2023)] Nauryz T., Briozzo A.C. \textit { Two-phase Stefan problem for generalized heat equation with nonlinear thermal coefficients.} Nonlinear Analysis: Real World Applications, \textbf{74}, (2023), 103944.

\bibitem[KhaNau(2021)]{18}[KhaNau(2021)] Kharin S.N., Nauryz T.A. \textit{Mathematical model of a short arc at the blow-off repulsion of electrical contacts during the transition from metallic phase to gaseous phase.} AIP Conference Proceedings, \textbf{2325},(2021), 020007.     

\bibitem[NauKha(2022)]{19}[NauKha(2022)] Nauryz T.A., Kharin S.N. \textit{Existence and uniqueness for one-phase spherical Stefan problem with nonlinear thermal coefficients and heat flux condition}, International Journal of Applied Mathematics, \textbf{35(5)}, (2022), 645-659.   

\bibitem[NaKh(2021)]{20} [NaKh(2021)] Kharin S.N., Nauryz T.A. One-phase spherical Stefan problem with temperature dependent coefficients. Eurasian Mathematical Journal, \textbf{12(1)}, (2021), pp. 49-56.   

\bibitem[BoBrKhNa(2024)]{21}[BoBrKhNa(2024)] Bollati J., Brizzo A.C., Kharin S.N., Nauryz T.A. \textit{Mathematical model of thermal phenomena of closure electrical contact with Joule heat source and nonlinear thermal coefficients}, Nonlinear Journal of Non-Linear Mechanics, \textbf{158}, (2024), 104568.

\bibitem[NauKas(2023)]{22}[NauKas(2023)] Nauryz T., Kassabek S.A., \textit{Mathematical modeling of the heat process in closure electrical contacts with a heat source.} International Journal of Non-Linear Mechanics, \textbf{157} (2023), 104533.

\bibitem[KasSur(2022)]{23}[KasSur(2022)] S. A. Kassabek,  D. Suragan,  \textit{Numerical approximation of the one-dimensional inverse Cauchy-Stefan problem using heat polynomials methods},  Computational and Applied Mathematics, \textbf{41:189}, (2022).

\bibitem[KasSur(2023)]{24}[KasSur(2023)] S. A. Kassabek,  D. Suragan,  \textit{A heat polynomials method for the two-phase inverse Stefan problem}, Computational and Applied Mathematics, \textbf{42:129}, (2023). 
}

\bibitem[KerIsm(2015)]{25}[KerIsm(2015)] N.B. Kerimov, M. I. Ismailov,  \textit{Direct and inverse problems for the heat equation with a dynamic-type boundary condition}, IMA Journal of Applied Mathematics, \textbf{80(5)} (2015), pp. 1519–1533,

\bibitem[DerSad(2024)]{26}[DerSad(2024)] B. Derbissaly, M. Sadybekov, \textit{ Inverse source problem for multi-term time-fractional di usion equation with nonlocal boundary conditions}, AIMS Mathematics, \textbf{9(4)} (2024), pp. 9969-9988.

\end
{thebibliography}
\end{document}